\documentclass[reqno]{amsart}
\usepackage[utf8]{inputenc}
\usepackage[T1]{fontenc}
\usepackage{lmodern}
\usepackage{hyperref}
\usepackage{amsmath}
\usepackage{amsfonts}
\usepackage{amssymb}
\usepackage{cancel}
\usepackage{color}
\usepackage{lipsum}
\numberwithin{equation}{section}
\newtheorem{theorem}{Theorem}[section]
\newtheorem{proposition}[theorem]{Proposition}
\newtheorem{remark}[theorem]{Remark}
\newtheorem{definition}{Definition}[section]
\newtheorem{claim}{Claim}[section]
\newtheorem{lemma}{Lemma}[section]
\newtheorem{example}{Example}[section]

\begin{document} 
\title{ Linear second order elliptic partial differential equations with Nonlinear Boundary Conditions 
at Resonance Without Landesman-Lazer Conditions }

\author[Alzaki Fadlallah\hfil ]
{Alzaki Fadlallah}

\address{Alzaki Fadlallah \hfill\break
Department of Mathematics, University of Alabama at
Birmingham \hfill\break
Birmingham, Alabama 35294-1170, USA }
\email{ zakima99@@math.uab.edu }

%\date

\begin{abstract}
We are concerned with the solvability of linear second order elliptic partial
differential equations with nonlinear boundary conditions at resonance, in which the nonlinear 
boundary conditions perturbation is not necessarily required to satisfy Landesman-Lazer conditions
or the monotonicity assumption. The nonlinearity may be unbounded. The nonlinearity interact, 
in some sense with the Steklov spectrum on boundary nonlinearity. 
The proofs are based on a priori estimates for possible solutions to a homotopy on
suitable trace and topological degree arguments. 

\end{abstract}
\maketitle

\section{Introduction}
This paper is concerned existence results for strong solutions of linear second order
elliptic partial differential  equations  with nonlinear boundary
condition at resonance of the form
\begin{equation}\label{E1}
\begin{gathered}
 -\Delta u + c(x)u = 0 \quad\text{in } \Omega,\\
\frac{\partial u}{\partial \nu}=\mu_{j}u+g(x,u)+h(x) \quad\text{on }
\partial\Omega,
\end{gathered}
\end{equation}

where $\Omega\subset\mathbb{R}^N$ , $N\geq 2$ is a bounded domain  with boundary
$\partial \Omega$ of class $C^2$, $c\in L^p(\Omega)$ with $p\ge N$ and  $c\ge 0$ a.e. on $\Omega$
with strict inequality on a set of positive measure, $\partial/\partial\nu :=\nu\cdot\nabla$
is the outward (unit) normal derivative on $\partial\Omega$, $\mu_j$ is $j^{th}$ eigenvalue of the problem 
\begin{equation}\label{E2}
\begin{gathered}
 -\Delta u + c(x)u = 0 \quad\text{in } \Omega,\\
\frac{\partial u}{\partial \nu}=\mu u \quad\text{on }
\partial\Omega,
\end{gathered}
\end{equation}

$g\colon\partial\Omega\times\mathbb{R}\rightarrow\mathbb{R}$
satisfies Carath\'{e}odory conditions
i,e.; 
\begin{description}
 \item[i] $g(.,u)$ is measurable on $\partial\Omega$, for each $u\in\mathbb{R}$,
\item[ii] $g(x,.)$ is continuous on $\mathbb{R},$ for $a.e.x\in \partial\Omega$,
\item[iii]for any constant $r>0$, there exists a function \\$\gamma_{r}\in L^{2}(\partial\Omega),$ such that 
\begin{equation}\label{00}
\begin{gathered}
|g(x,u)|\leq\gamma_{r}(x),
\end{gathered}
\end{equation}

for $a.e.x\in \Omega$, and all $u\in\mathbb{R}$ with $|u|\leq r,$
\end{description}

 and $h\in L^{2}(\partial\Omega).$ 
 By a (strong) solution to Eq.(\ref{E1}) we mean a function $u\in W^{2}_{p}(\Omega)$ satisfies Eq(\ref{E1}) (the second equality in Eq.(\ref{E1}) being 
 satisfied in the sense of trace). \\
 The paper is organized as follows:- In section 2 we study some of the properties of problem (\ref{E2}). In section 3 is devoted 
 the main results, and we illustrate our main theorem by giving an example of an unbounded nonlinear at boundary of $\Omega$
 doesn't satisfy Landesman-Lazer conditions at the boundary, non monotonicity assumption at the boundary. We conclude the paper 
 with some further results and  remarks. All the of our results are based upon Leray-Schauder continuation method and topological degree 
 
 \section{some of the properties of problem (\ref{E2})}

Let $H^{1}(\Omega)=W^{1,2}(\Omega)$
where $W^{1,2}(\Omega)$ is usual real Sobolev space of functions on $\Omega$ \\
Let defined the real inner-product as 
$$<u,v>_{c}:=\int_{\Omega}\triangledown u.\triangledown v+\int_{\Omega}c(x)uv~~~~\forall u,v\in  H^{1}(\Omega) $$
we proof that $<u,v>_{c}$ indeed inner-product 
{\color{blue}first we know that $H^{1}(\Omega) \hookrightarrow L^{P^*}(\Omega)$ where $P^*=\frac{2N}{N-2}$ then 
$$\int_{\Omega}c(x)u^2\leq^{ Holder~~ inequality}\left(\int_{\Omega}c(x)^p\right)^{\frac{1}{p}}\left(\int_{\Omega}u^{2q}\right)^{\frac{1}{q}}$$ 
where $\frac{1}{p}+\frac{1}{q}=1$, since $u\in L^{\frac{2N}{N-2}}(\Omega)$ this implies   that $q=\frac{N}{N-2}$, such that 
 $\left(\left(\int_{\Omega}u^{2\frac{N}{N-2}}\right)^{\frac{N-2}{2N}}\right)^2=||u||^{2}_{L^{\frac{2N}{N-2}}(\Omega)}$ so 
$\frac{1}{p}=1-\frac{1}{q}=1-\frac{N-2}{N}=\frac{2}{N},$ so that $p=\frac{N}{2}$, this implies that $c\in L^{\frac{N}{2}}(\Omega)$,  since $\Omega$ is bounded this implies that ( if $1\leq r < s\leq\infty$,
 then $L^{s}(\Omega)\subset L^{r}(\Omega)$) so that $c\in L^{p}(\Omega)$ for any $p\geq\frac{N}{2}$ for $N\geq 3$ ( when $N=2$, p=1). 
then $\int_{\Omega}c(x)u^2< \infty$  for all  $p\geq\frac{N}{2}$ and for all $u\in H^{1}(\Omega)$}
\begin{itemize}
 \item $$<u,v>_{c}=\int_{\Omega}\triangledown u.\triangledown v+\int_{\Omega}c(x)uv=\int_{\Omega}\triangledown v.\triangledown u+\int_{\Omega}c(x)vu=<v,u>_{c}~~~~\forall u,v\in  H^{1}(\Omega)$$
\item $$\forall a\in\mathbb{R}~~<au,v>_{c}=\int_{\Omega}\triangledown au.\triangledown v+\int_{\Omega}c(x)auv=a\left(\int_{\Omega}\triangledown u.\triangledown v+\int_{\Omega}c(x)uv\right)=a<u,v>_{c}$$
\item $$<u,u>_{c}=\int_{\Omega}|\triangledown u|^{2}+\int_{\Omega}c(x)u^2\geq 0$$, if $u\equiv 0$ this implies that $<u,u>_{c}=0$, if $<u,u>_{c}=0$ 
this implies that $\int_{\Omega}|\triangledown u|^{2}+\int_{\Omega}c(x)u^2=0$ there for $\int_{\Omega}|\triangledown u|^{2}=0$ and 
$\int_{\Omega}c(x)u^2=0$ since $c(x)> 0$ this implies that $u\equiv 0$ 
\end{itemize}
We will first study the spectrum that will be used for the comparison with nonlinearities in equation (\eqref{E2}). This spectrum include the Steklov (When $c=0$).\\ 

Consider the linear problem (Eq(\ref{E2}) 
The eigenproblem is to find a pair $(\mu,\varphi)\in\mathbb{R}\times H^{1}(\Omega) $ with $\varphi\not\equiv0$ such that 
\begin{equation}\label{3}
\begin{gathered}
\int_{\Omega}\triangledown \varphi.\triangledown v+\int_{\Omega}c(x)\varphi v=\mu\int_{\partial\Omega}\varphi v~~~~\forall\in H^{1}(\Omega) $$
\end{gathered}
\end{equation}
Now let $v=\varphi$, we see that if there such an eigenpair, then 
$\mu > 0$ and $\int_{\partial\Omega}\varphi^2>0$ since 
$$\int_{\Omega}|\triangledown \varphi|^2+\int_{\Omega}c(x)\varphi^2=\mu\int_{\partial\Omega}\varphi^2 $$
we know that $\varphi\not\equiv 0$ and $\int_{\Omega}c(x)dx >0$
(otherwise, $\varphi$ would be a constant function then we have that $\frac{1}{|\partial\Omega|}\int_{\Omega}c(x)dx=\mu$,   $\mu> 0$)
It is there fore a appropriate to consider the closed linear subspace of $H^{1}(\Omega)$
 defined by 
$$V(\Omega):=\left\{u\in H^{1}(\Omega):\int_{\partial\Omega}u^2=0~.i.e;~\Gamma u=0~ a.e ~{ on}~ \partial\Omega\right\}=H^{1}_{0}(\Omega)$$ where $\Gamma u$ denotes the trace of $u$ on $\partial\Omega$
and to look for the eigenfunctions associated with equation\eqref{E2} in the $c-$orthogonal complement $[V(\Omega)]^{\bot}=[H^{1}_{0}(\Omega)]^{\bot}$
of this subspace in $H^{1}(\Omega)$. Thus, one can split the Hilbert space $H^{1}(\Omega)$ as a direct $c-$orthogonal sum in the following way 
({\color{red} Since$\overline{C^{\infty}_{0}()(\Omega)}^{||.||_{H^{1}(\Omega)}}=H^{1}_{0}(\Omega)$ also we have that $Ker \Gamma=H^{1}_{0}(\Omega)$ i.e.; 
let $u_m\in H^{1}_{0}(\Omega) $ $u_{m}\to u $ in $H^{1}(\Omega)$ we will show that $u\in H^{1}_{0}(\Omega)$ since $\Gamma$ is conditions map you have that $\Gamma u_m\to \Gamma u$ since   $Ker \Gamma=H^{1}_{0}(\Omega)$ this implies that $\Gamma u=0$ $u=0$ on $\partial\Omega$ i.e.; $u_m\to 0$ on $\partial\Omega$ so
 $u\in H^{1}_{0}(\Omega) $ then $H^{1}_{0}(\Omega)$ is closed linear subspace of $H^{1}(\Omega)$ })
$$H^{1}(\Omega)=H^{1}_{0}(\Omega)\oplus_{c} [H^{1}_{0}(\Omega)]^{\bot} $$
Note also if $(\mu,\varphi)\in\mathbb{R}\times H^{1}(\Omega)$ is an eigenpair, then it follows from the definition of $H^{1}_{0}(\Omega$
that 
$$<\varphi,v>_{c}=\int_{\Omega}\triangledown \varphi\triangledown v+\int_{\Omega}c(x)\varphi v=0,~~~\forall v\in H^{1}_{0}(\Omega) ~{and}~\forall\varphi\in [H^{1}_{0}(\Omega)]^{\bot} $$
Besides the Sobolev space $H^{1}(\Omega)$, we shall make use in what follows the real Lebesgue space $L^{q}(\partial\Omega)$ for $1\leq q\leq\infty$, and 
of the continuity and compactness of the trace operator. 
$$\Gamma :H^{1}(\Omega)\to L^q(\partial\Omega)~~for~~1\leq q<\frac{2(n-1)}{n-2}$$
sometime we will just use $u$ in place of $\Gamma u$ when considering the trace of function on $\partial\Omega$ Throughout  this work we denote the $L^{2}(\partial\Omega)-$ inner product by 
$$<u,v>_{\partial}:=\int_{\partial\Omega}uv~~~and~~~||u||^{2}_{\partial}:=\int_{\partial\Omega}u^2~\forall u,v\in H^{1}(\Omega)$$
\newpage
\begin{definition}
Let $\mathbb{F}:H^{1}(\Omega)\to(-\infty,\infty]$ is functional,
then $\mathbb{F}$is said to be  G-differentiable at a point $u\in H^{1}(\Omega)$ if there is  a $\mathbb{F}'(u)$ such that 
$$\displaystyle\lim_{t\to 0}t^{-1}[\mathbb{F}(u+tv)-\mathbb{F}(u)]=\mathbb{F}'(u)v$$
With $\mathbb{F}'(u)$ a continuous  linear functional on  $H^{1}(\Omega)$ In this case, $\mathbb{F}'(u)$ is called the G-derivative of $\mathbb{F}$ at $u$ 
\end{definition}
\begin{theorem}\label{th}
Assume that $c$ as above. Then we have the following. 
\begin{description}
 \item[i] The eigenproblem \eqref{E2} has a sequence of real eigenvalues 
$$0<\mu_{1}\leq\mu_2\leq\mu_3\leq\cdots\leq\mu_j\leq\cdots\to\infty~as~j\to\infty$$ 
each eigenvalue has a finite-dimensional eigenspace. 
\item[ii] The eigenfunctions $\varphi_j$ corresponding to the eigenvalues $\mu_j$ from an $c-$orthogonal and $\partial-$orthonormal family in 
$[H^{1}_{0}(\Omega)]^{\bot}$ ( a closed subspace of $H^{1}(\Omega)$
\item[iii] The normalized eigenfunctions provide a complete $c-$orthonormal basis of $[H^{1}_{0}(\Omega)]^{\bot}$. Moreover, each function in $u\in [H^{1}_{0}(\Omega)]^{\bot}$
has a unique representation of the from 
\begin{equation}\label{4}
\begin{gathered}
u=\displaystyle\sum^{\infty}_{j=1}c_j\varphi_j~with~c_j:=\frac{1}{\mu_j}<u,\varphi_j>_c=<u,\varphi_j>_\partial \\
||u||^{2}_{c}=\displaystyle\sum^{\infty}_{j=1}\mu_j|c_j|^2
\end{gathered}
\end{equation}
In addition, $$||u||^{2}_{\partial}=\displaystyle\sum^{\infty}_{j=1}|c_j|^2$$
\end{description}
\end{theorem}
\begin{proof}
We will prove the existence of a sequence of real eigenvalues ${\mu_j}$ and the eigenfunctions $\varphi_j$ corresponding to the eigenvalues $\mu_j$ that form an orthogonal 
family in $[H^{1}_{0}(\Omega)]^{\bot}$\\
We will define the functionals 
\begin{equation}\label{5}
\begin{gathered}
\mathbb{I}:H^{1}(\Omega)\to[0,\infty)~~by\\
~~~~ \mathbb{I}(u):=\int_{\Omega}|\triangledown u|^2+\int_{\Omega}c(x)u^2=||u||^{2}_{c},~\forall u\in H^{1}(\Omega)\\
and\\ ~~\mathbb{J}:H^{1}(\Omega)\to[-1,\infty)~~ by\\
 \mathbb{J}(u):=\int_{\partial\Omega}u^2-1=||u||^{2}_{\partial}-1,   ~\forall u\in H^{1}(\Omega)
\end{gathered}
\end{equation}
Clearly $\mathbb{I}$ and $\mathbb{J}$ are $C^{1}-$ functional (i.e.; continuous differentiable) 
We compute $\mathbb{I}'(u)v~~\forall u,v\in H^{1}(\Omega)$  
\begin{equation}\label{5}
\begin{gathered}
\lim_{t\to 0}t^{-1}[\mathbb{I}(u+tv)-\mathbb{I}(u)]=\\
\lim_{t\to 0}t^{-1}\left[\int_{\Omega}|\triangledown u+tv|^2+\int_{\Omega}c(x)(u+tv)^2-\int_{\Omega}|\triangledown u|^2-\int_{\Omega}c(x)u^2\right]=\\
\small{\lim_{t\to 0}t^{-1}\left( \int_{\Omega}(|\triangledown u|^2+2t\triangledown u\triangledown v+t^2|\triangledown v|^2)+\int_{\Omega}(c(x)u^2+2tuv+t^2v^2)-
\int_{\Omega}|\triangledown u|^2-\int_{\Omega}c(x)u^2\right)}\\
\therefore\mathbb{I}'(u)v=2\left[\int_{\Omega}\triangledown u.\triangledown v+\int_{\Omega}c(x)uv\right]~~\forall u,v\in H^{1}(\Omega) \\
\end{gathered}
\end{equation}
Now you compute $\mathbb{J}'(u)v~~\forall u,v\in H^{1}(\Omega)$
\begin{equation}\label{6}
\begin{gathered}
\lim_{t\to 0}t^{-1}[\mathbb{J}(u+tv)-\mathbb{J}(u)]=\\
\lim_{t\to 0}t^{-1}\left[\int_{\partial\Omega}(u+tv)^2-\int_{\partial\Omega}u^2\right]=\\
\lim_{t\to 0}t^{-1}\left[\int_{\partial\Omega}u^2+2t\int_{\partial\Omega}uv+t^2\int_{\partial\Omega}v^2 -\int_{\partial\Omega}u^2\right]\\
\therefore\mathbb{J}'(u)v=2\int_{\partial\Omega}uv ~~\forall u,v\in H^{1}(\Omega)
\end{gathered}
\end{equation}
\begin{flushleft}
Claim 
\end{flushleft}
$\mathbb{I}'(u)v$ and $\mathbb{J}'(u)v$ are continuous functionals
\begin{proof} {Of the claim}
{\color{red}Let $u_{m}\to u$ in $H^{1}(\Omega)$,  we will show that $||\mathbb{I}'(u_m)-\mathbb{I}'(u)||_{\mathbb{L}(H^{1}(\Omega),\mathbb{R})}\to 0$,  as $m\to\infty$,  and
 $||\mathbb{J}'(u_m)-\mathbb{J}'(u)||_{\mathbb{L}(H^{1}(\Omega),\mathbb{R})}\to 0$, as $m\to\infty$, where ${\mathbb{L}(H^{1}(\Omega),\mathbb{R})}$  the set of all continuous functional from
$H^{1}(\Omega)$   to $\mathbb{R}$ since we know that 
$$||\mathbb{I}'(u_m)-\mathbb{I}'(u)||_{\mathbb{L}(H^{1}(\Omega),\mathbb{R})}=\displaystyle{\sup_{||v||_{c}=1}|\mathbb{I}'(u_m)v-\mathbb{I}'(u)v|,~~\forall u,v \in H^{1}(\Omega)}$$
$$|\mathbb{I}'(u_m)v-\mathbb{I}'(u)v|=2\left|\int_{\Omega}\triangledown u_m.\triangledown v+\int_{\Omega}c(x)u_mv-\int_{\Omega}\triangledown u.\triangledown v+\int_{\Omega}c(x)uv\right|$$
$$|\mathbb{I}'(u_m)v-\mathbb{I}'(u)v|\leq 2\left[\int_{\Omega}|\triangledown u_m-\triangledown u|.|\triangledown v|+\int_{\Omega}\sqrt{c(x)}\sqrt{c(x)}|u_m-u||v|\right]$$
$$\leq^{ Holder~~ inequality} 2\left[\left(\int_{\Omega}|\triangledown u_m-\triangledown u|^2\right)^{\frac{1}{2}}\left(\int_{\Omega}|\triangledown v|^2\right)^{\frac{1}{2}}+\left(\int_{\Omega}c(x)|u_m-u|^{2}\right)^{\frac{1}{2}}\left(\int_{\Omega}c(x)|v|^2\right)^{\frac{1}{2}}\right]$$
$$\leq 2||u_m-u||_{c}||v||_{c}\to 0 ~~as~~ m\to\infty$$ 
then, $||\mathbb{I}'(u_m)-\mathbb{I}'(u)||_{\mathbb{L}(H^{1}(\Omega),\mathbb{R})}\to 0$, as $m\to\infty$, 
so that $\mathbb{I}'(u)v$ is continuous functional.}

{\color{blue}Let $v_{m}\to v$ in $H^{1}(\Omega)$,  we will show that $|\mathbb{I}'(u)v_m-\mathbb{I}'(u)v|\to 0$,  as $m\to\infty$,  and
 $|\mathbb{J}'(u)v_m-\mathbb{J}'(u)v|\to 0$, as $m\to\infty$, 
 since we know that 
$$|\mathbb{I}'(u)v_m-\mathbb{I}'(u)v|=2\left|\int_{\Omega}\triangledown u.\triangledown v_m+\int_{\Omega}c(x)uv_m-\int_{\Omega}\triangledown u.\triangledown v+\int_{\Omega}c(x)uv\right|$$
$$|\mathbb{I}'(u)v_m-\mathbb{I}'(u)v|\leq 2\left[\int_{\Omega}|\triangledown v_m-\triangledown v|.|\triangledown u|+\int_{\Omega}\sqrt{c(x)}\sqrt{c(x)}|v_m-v||u|\right]$$
$$\leq^{ Holder~~ inequality} 2\left[\left(\int_{\Omega}|\triangledown v_m-\triangledown v|^2\right)^{\frac{1}{2}}\left(\int_{\Omega}|\triangledown u|^2\right)^{\frac{1}{2}}+\left(\int_{\Omega}c(x)|v_m-v|^{2}\right)^{\frac{1}{2}}\left(\int_{\Omega}c(x)|u|^2\right)^{\frac{1}{2}}\right]$$
$$\leq 2||v_m-v||_{c}||u||_{c}\to 0 ~~as~~ m\to\infty$$ 
then, $|\mathbb{I}'(u)v_m-\mathbb{I}'(u)v|\to 0$, as $m\to\infty$, 
so that $\mathbb{I}'(u)v$ is continuous functional.
similar argument we can prove that $\mathbb{J}'(u)v$ is  continuous functional. }
\end{proof}
This implies that $\mathbb{I},~~and~~\mathbb{J}$ are $C^{1}-$ functionals
 \end{proof}
We know that $\mathbb{I}$ is convex 
we will proof that 
$\forall t\in(0,1)$ and $\forall~u,v\in H^{1}(\Omega) $ 
we have that 

$$ \mathbb{I}(tu+(1-t)v)=||tu+(1-t)v||^{2}_{c}\leq \left(||tu||_{c}+||(1-t)v||_{c}\right)^{2} $$
$$\leq t^{2}||u||^{2}_{c}+2t(1-t)||u||_{c}||v||_{c}+(1-t)^2||v||^{2}_{c}\leq$$
$$t^{2}||u||^{2}_{c}+t(1-t)(||u||^{2}_{c}+||v||^{2}_{c})+(1-t)^2||v||^{2}_{c}\leq$$
$$t^{2}||u||^{2}_{c}+t||u||^{2}_{c}-t^{2}||u||^{2}_{c}+t||v||^{2}_{c}-t^{2}|v||^{2}+||v||^{2}_{c}-2t||v||^{2}_{c}+t^2||v||^{2}_{c}$$
 $$=t||u||^{2}_{c}+(1-t)||v||^{2}_{c}=t\mathbb{I}(u)+(1-t)\mathbb{I}(v)$$
So that $\mathbb{I}$ is convex functional.\\
We know that $\mathbb{I}$ is G-differentiable. \\
\textbf{ Claim:-}
 $\forall u,v\in H^{1}(\Omega)$ then $\mathbb{I}'(u)(v-u)\leq \mathbb{I}(v)-\mathbb{I}(u)$
\begin{proof}
IF $\mathbb{I}$ is convex, then 
$$\mathbb{I}(u+t(v-u))\leq\mathbb{I}(u)+t(\mathbb{I}(v)-\mathbb{I}(u))~~\forall u,v\in H^{1}(\Omega) ~\forall t\in (0,1) $$
$$\frac{\mathbb{I}(u+t(v-u))-\mathbb{I}(u)}{t}\leq \mathbb{I}(v)-\mathbb{I}(u) ~~\forall u,v\in H^{1}(\Omega) ~\forall t\in (0,1) $$
$$\lim_{t\to 0}\frac{\mathbb{I}(u+t(v-u))-\mathbb{I}(u)}{t}\leq \mathbb{I}(v)-\mathbb{I}(u) ~~\forall u,v\in H^{1}(\Omega) ~\forall t\in (0,1) $$
 so we have that 
$$\mathbb{I}'(u)(v-u)\leq \mathbb{I}(v)-\mathbb{I}(u)~~\forall u,v\in H^{1}(\Omega)$$
\end{proof}
\begin{theorem}
  Let $\mathbb{I}$ be G-differentiable and convex, then $\mathbb{I}$ is weakly lower-semi-continuous
\end{theorem}

\begin{proof}
 Since we have $u_{n}\rightharpoonup u $ in $H^{1}(\Omega)$ 
since $\mathbb{I}'$ is continuous then, $\lim_{n\to\infty}\mathbb{I}'(u)(u_{n})=\mathbb{I}'(u)(u)$
by the claim a above we have that 
$$\mathbb{I}'(u)(u_n-u)\leq \mathbb{I}(u_n)-\mathbb{I}(u)$$
so we have 
$$\liminf_{n\to\infty}\mathbb{I}'(u)(u_n-u)\leq \liminf_{n\to\infty}\left(\mathbb{I}(u_n)-\mathbb{I}(u)\right)$$
since the limit of the left hand side exist and equal zero then we have that 
$$0\leq\liminf_{n\to\infty}\left(\mathbb{I}(u_n)-\mathbb{I}(u)\right)$$
so we have 
$$\mathbb{I}(u) \leq \liminf_{n\to\infty}\mathbb{I}(u_n)$$
therefore $\mathbb{I}$ is weakly lower-semi-conditions
\end{proof}
When $N=2$ we know that $H^{1}(\Omega)\hookrightarrow L^{q}(\Omega)$ when $q\in [2,\infty)$
let $u\in H^{1}(\Omega)$ then $u\in L^{q}(\Omega)$ when $q\in [2,\infty)$, by {H\"{o}lder ~~inequality}
$$\int_{\Omega}c(x)u^2\leq \left(\int_{\Omega}c(x)^{p}\right)^\frac{1}{p}\left(\int_{\Omega}u^{2r}\right)^\frac{1}{r}$$ 
 where $\frac{1}{p}+\frac{1}{r}=1$  so $2r=q\Rightarrow r=\frac{q}{2} ~ \forall q\in[2,\infty)$ 
then $1\leq r<\infty$
and $p=\frac{r}{r-1}>1$ \\
For $u,v\in H^{1}(\Omega)$. Now we show that $\mathbb{I}$ attains its minimum on the constraint set 
$W_0=\{u\in [H^{1}_{0}(\Omega)]^{\bot}:\mathbb{J}(u)=0\}.$
Let $\alpha=\displaystyle\inf_{u\in W_{0}}\mathbb{I}(u)$, by using the continuity of the trace operator, the Sobolev embedding theorem and the lower semi-continuity 
of $\mathbb{I}$\\ Let $\{u_{n}\}_{n\geq 1}$ be a minimizing sequence in $W_0$ for $\mathbb{I}$ 
since $\lim_{n\to\infty}\mathbb{I}(u_n)=\alpha$, we know  that  $\mathbb{I}(u_{n})=||u_n||^{2}_{c}$ by the definition of $\alpha$ we have that 
for all sufficiently large $n$, and for all $\epsilon >0$, then $||u_n||^{2}_{c}\leq \alpha+\epsilon$ by using the equivalent norm we have that 
there is exist $\beta$ such that 
$$||u_n||^{2}_{H^{1}(\Omega)}\leq\beta ||u_n||^{2}_{c}$$ so we have that 

$$||u_n||^{2}_{H^{1}(\Omega)}\leq\beta ||u_n||^{2}_{c}\leq\beta(\alpha+\epsilon),$$
so this sequence is bounded in $H^{1}(\Omega)$. Thus it has a weakly convergent subsequence $\{u_{n_j}:j\geq 1\}$ which convergent 
weakly to limit $\hat{u}$ in $H^{1}(\Omega)$. From Rellich's theorem this subsequence convergent strongly to $\hat{u}$ in $L^{2}(\Omega)$
so $\hat{u}$ in $W_0$. Thus $\mathbb{I}(\hat{u})=\alpha$ as the functional is weakly l.s.c.\\

 Then there exists $\varphi_1$ such that $\mathbb{I}(\varphi_1)=\alpha$.  Hence, $\mathbb{I}$ attains
 its minimum at $\varphi_1$ and $\varphi_1$ satisfies the following 
\begin{equation}\label{8}
\begin{gathered}
 \int_{\Omega}\triangledown \varphi_1\triangledown v+\int_{\Omega}c(x)\varphi_{1}v=\mu_{1}\int_{\partial\Omega}\varphi_{1}v
\end{gathered}
\end{equation}
For all $v\in[H^{1}_{0}(\Omega)]^{\bot} $ We see that $(\mu_{1},\varphi_{1})$ satisfies (\ref{3}) and $\varphi_{1} \in W$ this implies that 
$\varphi_{1}\in [H^{1}_{0}(\Omega)]^{\bot} $ by the definition of $W$, Now take $v=\varphi_{1}$ in (\ref{8}), we obtain that the eigenvalue
$\mu_{1}$ is the infimum $\alpha=\mathbb{I}(\varphi_{1})=\mu_{1}$. This means that we could define $\mu_{1}$ by Rayleigh quotient 
$$\mu_{1}=\inf_{\substack{u\in H^{1}(\Omega)\\u\neq 0}}\frac{\mathbb{I}(u)}{||u||^{2}_{\partial}}$$
Clearly, $\mu_{1}=\mathbb{I}(\varphi_{1})\geq 0$. Indeed  assume that $\mathbb{I}(\varphi_{1})=0$ then $|\triangledown \varphi_{1}|=0$ on $\Omega$
, hence $\varphi_{1}$ must be a constant that contradicts the assumptions imposed on $c(x)$. Thus $\mu_{1} >0$. \\

Now we show the existence of higher eigenvalues. \\
Define 
$$\mathbb{F}:W_{0}\to\mathbb{R}~by~\mathbb{F}(u)=<u,\varphi_{1}>_{\partial}$$ 
we know that the kernel of $\mathbb{F}$ 
$$ker\mathbb{F}=\{u\in W_{0}:\mathbb{F}(u)=0,~ i.e.; <u,\varphi_{1}>_{\partial}=0\}=:W_{1}.$$
Since $W_{1}$ is the null-space of the continuous functional $<.,\varphi_{1}>_{\partial}$ on $[H^{1}_{0}(\Omega)]^{\bot},$ 
$W_{1}$ is a closed subspace of $[H^{1}_{0}(\Omega)]^{\bot}$, and it is therefore a Hilbert space itself under the same inner product 
$<.,.>_{c}$. Now we define 
$$\mu_{2}=\inf\{\mathbb{I}(u):u\in W_{1}\}=\inf_{\substack{u\in W_{1}\\u\neq 0}}\frac{\mathbb{I}(u)}{||u||^{2}_{\partial}}$$
Since $W_{1}\subset W_0$ then we have that $\mu_{1}\leq\mu_{2}$. Moreover, we can repeat the above arguments to show that 
$\mu_{2}$ is achieved at some $\varphi_{2}\in [H^{1}_{0}(\Omega)]^{\bot}.$ \\
We let 
$$W_{2}=\{u\in W_{1}:<u,\varphi_{2}>_{\partial}=0\},$$
and 
$$\mu_{3}=\inf\{\mathbb{I}(u):u\in W_{2}\}=\inf_{\substack{u\in W{2}\\u\neq 0}}\frac{\mathbb{I}(u)}{||u||^{2}_{\partial}}$$
Since $W_{2}\subset W_{1}$ then we have that $\mu_{2}\leq\mu_{3}$. Moreover, we can repeat the above arguments to show that 
$\mu_{3}$ is achieved at some $\varphi_{3}\in [H^{1}_{0}(\Omega)]^{\bot}.$ \\
Proceeding inductively, we let 
$$W_{j}=\{u\in W_{j-1}:<u,\varphi_{j}>_{\partial}=0\},~\forall j\in\mathbb{N}$$ 
and 
$$\mu_{j+1}=\inf\{\mathbb{I}(u):u\in W_{j}\}=\inf_{\substack{u\in W{j}\\u\neq 0}}\frac{\mathbb{I}(u)}{||u||^{2}_{\partial}}$$
In this way, we generate a sequence of eigenvalues 
$$0<\mu_{1}\leq\mu_{2}\leq\mu_{3}\leq\ldots\leq\mu_{j}\leq\ldots$$
whose associated $\varphi_{j}$ are $c-$orthogonal and $\partial-$orthonormal in $[H^{1}_{0}(\Omega)]^{\bot}$ \\ 
\textbf{Claim 1} $\mu_{j}\to\infty$ as $j\to\infty$
\begin{proof}
 Suppose by contradiction that the sequence is bounded above by constant. Therefore, the corresponding sequence of eigenfunctions $\varphi_{j}$
is bounded in $H^{1}(\Omega)$ (i.e.; by the definition of the limit at $\infty$ $\forall M>0,~\exists N>0$ such that $|\varphi_{j}|>M,$ whenever $j>N$, 
the  ingation of the stetment $\exists M>0$ such that $|\varphi_{j}|\leq M~\forall j$). 
By Rellich-Kondrachov theorem and the compactness of the trace operator, there is a Cauchy subsequence (which we again denote by $\varphi_{j}$
such that 
\begin{equation}\label{9}
\begin{gathered}
||\varphi_{j}-\varphi_{k}||^{2}_{\partial}\to 0.
\end{gathered}
\end{equation}

Since the $\varphi_{j}$ are $\partial-$orthonormal, we have that 
$||\varphi_{j}-\varphi_{k}||^{2}_{\partial}=||\varphi_{j}||^{2}_{\partial}+||\varphi_{k}||^{2}_{\partial}=2>0$, 
if$j\neq k,$ which contradicts 
(\ref{9}). Thus, $\mu_{j}\to\infty.$ we have that each $\mu_{j}$ accurs only finitely many times. 
\end{proof}
\textbf{Claim 2} Each eigenvalue $\mu_{j}$ has a finite-dimensional eigenspace. \\ 
\begin{proof}
 Suppose by contradiction that each eigenvalue $\mu_{j}$ has infinite-dimensional eigenspace. 
let $\mu$ has corresponding sequence of eigenfunctions $\{\varphi_{1},\varphi_{2},...,\varphi_{j},...\}$ 
we know that $\mu=||\varphi_{1}||^{2}_{c}=...=||\varphi_{j}||^{2}_{c}=...$, this contradicts  claim 1
therefore, each eigenvalue has a finite-dimensional eigenspace

\end{proof}
We will show that the normalized eigenfunctions provide a complete orthonormal basis of $[H^{1}_{0}(\Omega)]^{\bot}$. 
Let $$\psi_{j}=\frac{1}{\sqrt{\mu_{j}}}\varphi_{j},$$
so that $||\psi_{j}||^{2}_{c}=1$ \\
\textbf{Claim 3} 
The sequence $\{\psi_{j}\}_{j\geq1}$ is a maximal $c-$orthonormal family of $[H^{1}_{0}(\Omega)]^{\bot}$. (we know that  the set maximal $c-$orthonormal 
if and only if it is  complete orthonormal basis)
\begin{proof}
 Suppose by contradiction that the sequence $\{\psi_{j}\}_{j\geq1}$ is not maximal, then there exists a $\xi\in [H^{1}_{0}(\Omega)]^{\bot}$ and 
$\xi\not\in \{\psi_{j}\}_{j\geq1}$, 
such that $||\xi||^{2}_{c}=1$ and $<\xi,\psi_{j}>_{c}=0~\forall j$, i.e.;
 $$0=<\xi,\psi_{j}>_{c}=<\xi,\frac{1}{\sqrt{\mu_{j}}}\varphi_{j}>_{c}=$$$$\frac{1}{\sqrt{\mu_{j}}}<\xi,\varphi_{j}>_{c}=^{(~by~\ref{8})}\frac{\mu_{j}}{\sqrt{\mu_{j}}}<\xi,\varphi_{j}>_{\partial}=
\mu_{j}<\xi,\frac{1}{\sqrt{\mu_{j}}}\varphi_{j}>_{\partial}={\mu_{j}}<\xi,\psi_{j}>_{\partial},$$  since $\mu_{j}>0~\forall j$. Therefore 
$<\xi,\psi_{j}>_{\partial}=0$. We have that $\xi\in W_{j}~\forall j\geq 1$. It follows from the definition of $\mu_{j}$ that 
$$\mu_{j}\leq\frac{||\xi||^{2}_{c}}{||\xi||^{2}_{\partial}}=\frac{1}{||\xi||^{2}_{\partial}}~\forall ~j\geq 1.$$ 
Since we know from claim 1 that $\mu_{j}\to\infty$ we have that $||\xi||^{2}_{\partial}=0$, therefore $\xi=0$a.e in $\Omega$,
which condradicts  the definition of $\xi.$ Thus the sequence $\{\psi_{j}\}_{j\geq1}$ 
is a maximal $c-$orthonormal family of $[H^{1}_{0}(\Omega)]^{\bot},$ so the sequence $\{\psi_{j}\}_{j\geq1}$  provides a complete orthonormal basis of $[H^{1}_{0}(\Omega)]^{\bot};$
that is, for any $u\in[H^{1}_{0}(\Omega)]^{\bot}$,\\ \\
$u=\displaystyle\sum^{\infty}_{j=1}d_{j}\psi_{j}$  with $d_{j}=<u,\psi_{j}>_{c}=\frac{1}{\sqrt{\mu_{j}}}<u,\varphi_{j}>_{c},$ and $||u||^{2}_{c}=\displaystyle\sum^{\infty}_{j=1}|d_{j}|^2$\\ \\

$$u=\displaystyle\sum^{\infty}_{j=1}d_{j}\frac{1}{\sqrt{\mu_{j}}}\varphi_{j},$$ now let 
$$c_{j}=d_{j}\frac{1}{\sqrt{\mu_{j}}}=\frac{1}{\mu_{j}}<u,\varphi_{j}>_{c}=^{(\ref{8})}<u,\varphi_{j}>_{\partial}.$$
Therefore, 
$$u=\displaystyle\sum^{\infty}_{j=1}c_{j}\varphi_{j},$$ and 
$$||u||^{2}_{c}=\displaystyle\sum^{\infty}_{j=1}|c_{j}|^{2}||\varphi_{j}||^{2}_{c}=\displaystyle\sum^{\infty}_{j=1}\mu_{j}|c_{j}|^{2}$$
\end{proof}
\textbf{Claim 4} 
We shall show that 
$$||u||^{2}_{\partial}=\displaystyle\sum^{\infty}_{j=1}|c_{j}|^{2}$$
\begin{proof}
$$||u||^{2}_{\partial}=<u,u>_{\partial}=
<\displaystyle\sum^{\infty}_{j=1}c_{j}\varphi_{j},\displaystyle\sum^{\infty}_{k=1}c_{k}\varphi_{k}>_{\partial}
=\displaystyle\sum^{\infty}_{j=1}c_{j}\displaystyle\sum^{\infty}_{k=1}c_{k}<\varphi_{j},\varphi_{k}>_{\partial}
=\displaystyle\sum^{\infty}_{j=1}|c_{j}|^{2}.$$
Thus $$||u||^{2}_{\partial}=\displaystyle\sum^{\infty}_{j=1}|c_{j}|^{2}$$
\end{proof}
The following result gives a variational chararcterization of the eigenvalues and a splitting of the space 
$[H^{1}_{0}(\Omega)]^{\bot}$ (and, hence, of $H^{1}(\Omega)$ which will be needed in the proofs of the result on nonlinear problems.\\ \\
\textbf{Corollary 1}
 Assume that $c$ satisfy the above condition. Then we have the following.
\begin{description}
 \item[i] For all $u\in H^{1}(\Omega) ,$
\begin{equation}\label{10}
\begin{gathered}
\mu_{1}\int_{\partial\Omega}u^2\leq\int_{\Omega}|\triangledown u|^2+\int_{\Omega}c(x)u^2,
\end{gathered}
\end{equation}
where $\mu_{1} >0$ is the least Steklov eigenvalue for equation (\ref{E2}). If equality holds in (\ref{10}), then $u$ is a multiple of 
an eigenfunction of equation (\ref{E2}) corresponding to $\mu_{1}$
 \item[ii] For every $v\in\oplus_{i\leq j}E(\mu_{i}),$ and $w\in\oplus_{i\geq j+1}E(\mu_{i}),$ we have that 
\begin{equation}\label{11}
\begin{gathered}
||v||^{2}_{c}\leq\mu_{j}||v||^{2}_{\partial} ~~and~~||w||^{2}_{c}\geq\mu_{j+1}||w||^{2}_{\partial} 
\end{gathered}
\end{equation}
where $E(\mu_{i})$ is the $\mu_{i}$-eigenspace and $\oplus_{i\leq j}E(\mu_{i})$ is span of the eigenfunctions associated 
to eigenvalues up to $\mu_{j}$
\end{description}
\begin{proof}
 If $u=0$, then the inequality (\ref{10}) holds. otherwise, if  $0\neq u\in H^{1}(\Omega),$
then $u=u_{1}+u_{2},$ where $u_{1}\in[H^{1}_{0}(\Omega)]^{\bot} $, and $u_{2}\in H^{1}_{0}(\Omega).$ Therefore, by the $c-$orthogonality, and the characterization of
 $\mu_{1}$ (i.e.; $\mu_{1}||u_{1}||^{2}_{\partial}\leq||u_{1}||^{2}_{c})$ we get that 
$$\mu_{1}||u||^{2}_{\partial}=\mu_{1}(||u_{1}||^{2}_{\partial}+||u_{2}||^{2}_{\partial}\leq||u_{1}||^{2}_{c}+||u_{2}||^{2}_{c}=||u||^{2}_{c}$$.
Thus, the inequality (\ref{10}) holds. \\
Now assume we have that 
$$||u||^{2}_{c}=\mu_{1}||u||^{2}_{\partial}\implies\mu_{1}= \frac{||u||^{2}_{c}}{||u||^{2}_{\partial}}$$
we know that $\mu_{1}=\frac{||\varphi_{1}||^{2}_{c}}{||\varphi_{1}||^{2}_{\partial}}$ where $\varphi_{1}$ the eigenfunction corresponding to $\mu_{1}$,
therefore, $u$ is a multiple of an eigenfunction of equation (\ref{E2}) corresponding to $\mu_{1}$\\
The inequalities (\ref{11}) by $\ref{th}$ we have that 
$$||v||^{2}_{c}=\displaystyle\sum^{\infty}_{j=1}\mu_j|c_j|^2~~\forall v\in\oplus_{i\leq j}E(\mu_{i})$$.
Now let  $\mu_{j}=\max\mu~\forall i\leq j,$ then we have that 
$$||v||^{2}_{c}=\displaystyle\sum^{\infty}_{j=1}\mu_j|c_j|^2\leq\max\mu\displaystyle\sum^{\infty}_{j=1}|c_j|^2=\mu_{j}||v||^{2}_{\partial}~\forall~v\in\oplus_{i\leq j}E(\mu_{i})$$

$$||w||^{2}_{c}=\displaystyle\sum^{\infty}_{j=1}\mu_j|c_j|^2~~\forall v\in\oplus_{i\leq j}E(\mu_{i})$$.
Now let  $\mu_{j+1}=\min\mu~\forall i\geq j+1,$ then we have that 
$$||w||^{2}_{c}=\displaystyle\sum^{\infty}_{j=1}\mu_j|c_j|^2\geq\min\mu\displaystyle\sum^{\infty}_{j=1}|c_j|^2=\mu_{j+1}||w||^{2}_{\partial}~\forall~w\in\oplus_{i\geq j+1}E(\mu_{i})$$

\end{proof}
The following proposition shows the principality of the first eigenvalue $\mu_{1}.$
\begin{proposition}\label{p1}
The first eigenvalue $\mu_{1}$ is simple if and only if the associated eigenfunction $\varphi_{1}$ does not changes sign (i.e.; 
$\varphi_{1}$ is strictly positive or strictly negative in $\Omega$ .
\end{proposition}
\begin{proof}
 Assume that the first eigenvalue $\mu_{1}$ is simple, we will show that  associated eigenfunction $\varphi_{1}$ does not changes sign in $\Omega$, suppose it does and 
let $\varphi_{1}=\varphi^{+}_{1}+\varphi^{-}_{1},$ where $\varphi^{+}_{1}=\max\{\varphi_{1},0\},$ and $\varphi^{-}_{1}=\min\{0,\varphi_{1}\}$\\ 
If $\varphi_{1}\in H^{1}(\Omega)$. Then $\varphi^{+}_{1},\varphi^{-}_{1}\in H^{1}(\Omega)$
proof of that we know that $\varphi^{+}_{1}=\frac{1}{2}(\varphi_{1}+|\varphi_{1}|)$
clearly $\varphi^{+}_{1}\in L^{2}(\Omega)$, define 
$$V_{\epsilon}=(\varphi_{1}^{2}+\epsilon^{2})^{\frac{1}{2}}-\epsilon$$ 
$$|\varphi_{1}|=\displaystyle\lim_{\epsilon\to 0}V_{\epsilon},$$ we will show that 
$$D^{i}V_{\epsilon}=\frac{\varphi_{1}}{(\varphi_{1}^{2}+\epsilon^{2})^{\frac{1}{2}}}D_{i}\varphi_{1}\longrightarrow^{L^{2}(\Omega)}~sign D^{i}\varphi_{1}$$
$\forall\epsilon >0$, then $0 \leq V_{\epsilon}\leq |\varphi_{1}|$
since
 $$V_{\epsilon}^{2}=\left((\varphi_{1}^{2}+\epsilon^{2})^{\frac{1}{2}}-\epsilon\right)^{2}=\varphi_{1}^{2}+\epsilon^{2}-2(\varphi_{1}^{2}+\epsilon^{2})^{\frac{1}{2}}\epsilon+\epsilon^{2}
=\varphi_{1}^{2}+2\epsilon(\epsilon-(\varphi_{1}^{2}+\epsilon^{2})^{\frac{1}{2}}\leq\varphi_{1}^{2},$$
therefore  $V_{\epsilon}\leq |\varphi_{1}|$, 
$$\displaystyle\lim_{\epsilon\to 0}||\frac{\varphi_{1}}{(\varphi_{1}^{2}+\epsilon^{2})^{\frac{1}{2}}}D_{i}\varphi_{1}-sign D^{i}\varphi_{1}||_{L^{2}(\Omega)}=0$$
Therefore, 
$$\frac{\varphi_{1}}{(\varphi_{1}^{2}+\epsilon^{2})^{\frac{1}{2}}}D_{i}\varphi_{1}\longrightarrow^{L^{2}(\Omega)}~sign D^{i}\varphi_{1}$$

Thus, $\varphi_{1}^{+}\in H^{1}(\Omega)$, similar $\varphi_{1}^{-}\in H^{1}(\Omega)$\\
By the characterization of $\mu_{1}$ it follows that  
$$<\varphi_{1},\varphi_{1}>_{c}=\mu_{1}<\varphi_{1},\varphi_{1}>_{\partial},$$ since  $\varphi_{1}^{+}\in H^{1}(\Omega)$,and  $\varphi_{1}^{-}\in H^{1}(\Omega),$ we have that 
$$\mu_{1}<\varphi_{1}^{+},\varphi_{1}^{+}>_{\partial}\leq<\varphi_{1}^{+},\varphi_{1}^{+}>_{c},$$
$$\mu_{1}<\varphi_{1}^{-},\varphi_{1}^{-}>_{\partial}\leq<\varphi_{1}^{-},\varphi_{1}^{-}>_{c}.$$ Therefore
$$0\leq~ <\varphi_{1}^{+},\varphi_{1}^{+}>_{c}+<\varphi_{1}^{-},\varphi_{1}^{-}>_{c}- \mu_{1}<\varphi_{1}^{+},\varphi_{1}^{+}>_{\partial}
-\mu_{1}<\varphi_{1}^{-},\varphi_{1}^{-}>_{\partial}=$$$$
<\varphi_{1}^{+}+\varphi_{1}^{-},\varphi_{1}^{+}+\varphi_{1}^{-}>_{c}- \mu_{1}<\varphi_{1}^{+}+\varphi_{1}^{-},\varphi_{1}^{+}+\varphi_{1}^{-}>_{\partial}
=<\varphi_{1},\varphi_{1}>_{c}-\mu_{1}<\varphi_{1},\varphi_{1}>_{\partial}=0.$$  
It follows that $\varphi_{1}^{+},$ and $\varphi_{1}^{-}$ are also eigenfunctions corresponding to $\mu_{1}$ we have that
  $\varphi_{1}^{+}>0~a.e~in~\Omega,$ and $\varphi_{1}^{-}<0~a.e~in~\Omega,$ which is impossible since $\mu_{1}$ it is simple. 
Thus $\varphi_{1}$ does not change sign in $\Omega.$\\
Assume $\varphi_{1}$ change sig, then $\varphi_{1}^{+},$ and $\varphi_{1}^{-}$ are also eigenfunctions corresponding to $\mu_{1}$ and 
they are linearly independent. Hence, $\mu_{1}$  is not simple.
On the other hand, suppose that $\mu_{1}$ is not simple, and let $\varphi$ and $\psi$ be two eigenfunctions corresponding 
to $\mu_{1}$ they are linearly independent. If $\varphi$ or  $\psi$ changes sign, then the proposition is proved. Otherwis, 
supposing without loss of generality that $\varphi$ and   $\psi$ positive, we will prove that there exists $a\in\mathbb{R}$ such that 
the eigenfunction (corresponding to $\mu_{1}$) $\varphi+a\psi$ changes sign. Indeed, suppose that, for all $\alpha\in\mathbb{R},$ 
$\varphi+\alpha\psi$ does not change.\\
Let the function $h:\mathbb{R}\to\mathbb{R}$ be define by 
$$h(\alpha)=\int\varphi+\alpha\int\psi.$$
Since $h$ is continuous, there exists $a\in\mathbb{R}$ such that 
$$h(a)=\int\varphi+a\int\psi=0.$$
Hence,  which contradicts the fact  $\varphi$ and $\psi,$ are linearly independent. 
Thus, $\varphi+a\psi,$ changes sign. The proof is complete. 
\end{proof}
\begin{remark}
Note that if we have smooth data and $\partial\Omega$ in \ref{p1}, then the eigenfunction $\varphi_{1}(x)$ on $\partial\Omega$ as well, 
by the boundary point lemma (see for example Evans).
\end{remark}
\section{the main results}
 \begin{theorem}\label{thm1}
Assume that 
\begin{equation}\label{01}
\begin{gathered}
g(x,u)u\geq 0,
\end{gathered}
\end{equation}

for $a.e.;~x\in \partial\Omega$, and all $u\in\mathbb{R}.$ Moreover, suppose that for all constant $\sigma>0,$
there exist a constant $K=K(\sigma),$ and function $b=b(\sigma)\in L^{\infty}(\partial\Omega)$ such that

\begin{equation}\label{02}
\begin{gathered}
|g(x,u)|\leq(\Gamma(x)+\sigma)|u|+b(x),
\end{gathered}
\end{equation}

for $a.e.;~x\in \partial\Omega$, and all $u\in\mathbb{R},$ with $|u|\geq K,$ where 
$\Gamma\in L^{\infty}(\partial\Omega),$ such that for $a.e.;~x\in \partial\Omega$

\begin{equation}\label{03}
\begin{gathered}
0\leq\Gamma(x)\leq (\mu_{j+1}-\mu_{j}),~~j\in\mathbb{N}
\end{gathered}
\end{equation}

(where ($\mu_{j+1}-\mu_{j})$ the $({j+1})^th$ Steklov eigenvalue ($c=0$) ) with $\Gamma(x)<(\mu_2-\mu_{j}),$ on a subset of $\partial\Omega$ of positive measure. \\
Then, equation (\ref{E1}) has least one solution $u\in W^{2}_{p}(\Omega)$ for any $h\in L^{2}(\partial\Omega),$ with 

\begin{equation}\label{04}
\begin{gathered}
\int_{\partial\Omega}h(x)\varphi_j(x)dx=0
\end{gathered}
\end{equation}

where $\varphi_{j}$ the ${j^{th}}$ eigenfunction of (\ref{E2})
 \end{theorem}
By the solution of equation (\ref{E1}) we mean a function  $u\in W^{2}_{p}(\Omega),$ which satisfies the differential equation $a.e.$
To prove theorem (\ref{thm1}) we shall need to three useful lemmas stated and proved below \\

We define the linear (Steklov when $c=0$) boundary open 
$$L:Dom(L)\subset W^{2}_{p}(\Omega)\Subset H^{\frac{1}{2}}(\partial\Omega)\to H^{\frac{1}{2}}(\partial\Omega) $$
by 
$$Lu:=\frac{\partial u}{\partial\nu}-\mu_{j}u,$$
 where
$$Dom(L):=\{u\in W^{2}_{p}(\Omega):-\Delta u+c(x)u=0\}$$
We denote by $N(L)$ the nullspace of $L$ and $R(L)$ closed range see \cite{MN}, 
and we observe that 
$$R(L)=(N(L))^{\perp}$$
which implies that the right inverse of $L$ defined by 
$$K=(\/Dom(L)\cap R(L))^{-1}:R(L)\to R(L)$$
is  well defined continuous linear operator and $K$ is compact (the proof similar proof in \cite{MN}).
denoting by $P_{j}$ the orthogonal projection onto the eigenspace $N(L-\mu_{j}I)$ where ($I$ is identity map), $L$
admits the spectral spectral representation 
$$L=\displaystyle\sum_{j=1}^{\infty}\mu_{j}P_{j}$$
For each $u\in H^{1}(\partial\Omega),$  (in Trace sense) let us write 
$$u(x)=\overline{u}(x)+u^{0}(x)+\widetilde{u}(x),~~\forall ~\in\partial\Omega$$
 where, if the Fourier expansion of $u$ ( see theorem \ref{th})
 $$u=\displaystyle\sum_{j=1}^{\infty}P_{j}u $$
 then
 $$\overline{u}=\displaystyle\sum_{1\leq j<N}P_{j}u$$$$u^{0}=P_{N}u$$$$\widetilde{u}=\displaystyle\sum_{N<j<\infty}P_{j}u$$
so that, with obvious notations 

$ H^{1}(\partial\Omega)= \overline{H}^{1}(\partial\Omega)\bigoplus \mathring{H}^{1}(\partial\Omega)\bigoplus \widetilde{H}^{1}(\partial\Omega)$.\\
Moreover, we shall use the notation $u^{\bot}=u-u^{0}$
\begin{lemma}\label{le1}
Let $\Gamma\in L^{\infty}(\partial\Omega)$ be such that for $a.e.x\in \partial\Omega$, $0\leq\Gamma(x)\leq (\mu_{j+1}-\mu_{j}),$ with $\Gamma(x)<(\mu_{j+1}-\mu_{j}),$ 
 on a subset of $\partial\Omega$ of positive measure,  with
 
 \begin{equation}\label{13}
\begin{gathered}
\int_{\partial\Omega}(\mu_{j+1}-\mu_{j})\varphi_{j+1}^{2}(x)dx>0
\end{gathered}
\end{equation}
 
 for all $\varphi_{j+1}\in N(L-\mu_{j+1}I),~~\varphi_{j+1}\neq 0~i.e.; L\varphi_{j+1}(x)=\mu_{j+1}\varphi_{j+1}(x)$ 
 eigenfunction corresponding to the eigenvalue $\mu_{j+1}$

 Then there exists a constant $\delta=\delta(\Gamma)>0,$ such that for all
 $u\in {H}^{1}(\partial\Omega),$ one has 
 $$D_{\Gamma}(u):=\langle Lu-(\mu_{j}+\Gamma)u,\widetilde{u}-(\overline{u}+u^{0})\rangle_{\partial}\geq\delta||{u^{\bot}}||^{2}_{H^{1}}(\partial\Omega)$$
 \end{lemma}
 \begin{proof}
 Taking into account the orthogonality of $\overline{u}+u^{0}$ with respect to $\widetilde{u}$ and the fact that $u^{0}\in N(L-\mu_{j}I)$  one has 
 $$D_{\Gamma}(u)=\langle L(\overline{u}+u^{0}+\widetilde{u})-(\mu_{j}+\Gamma)(\overline{u}+u^{0}+\widetilde{u}),\widetilde{u}-(\overline{u}+u^{0})\rangle_{\partial}$$
 $$=\langle L\overline{u}+Lu^{0}+L\widetilde{u}-(\mu_{j}+\Gamma)(\overline{u}+\widetilde{u})-\mu_{j}u^{0}-\Gamma u^{0},\widetilde{u}-(\overline{u}+u^{0})\rangle_{\partial}$$
$$=\langle L\overline{u}+\cancel{(L-\mu_{j}I)u^{0}}+L\widetilde{u}-(\mu_{j}+\Gamma)(\overline{u}+\widetilde{u})-\Gamma u^{0},\widetilde{u}-(\overline{u}+u^{0})\rangle_{\partial}$$
 $$=\langle L\overline{u}+L\widetilde{u}-(\mu_{j}+\Gamma)(\overline{u}+\widetilde{u})-\Gamma u^{0},\widetilde{u}-(\overline{u}+u^{0})\rangle_{\partial}$$
 $$=\langle L\overline{u}+L\widetilde{u}-(\mu_{j}+\Gamma)(\overline{u}+\widetilde{u}),\widetilde{u}-(\overline{u}+u^{0})\rangle_{\partial}-\langle\Gamma u^{0},\widetilde{u}-(\overline{u}+u^{0})\rangle_{\partial}$$
 $$=\langle L\widetilde{u}-(\mu_{j}+\Gamma)(\widetilde{u}),\widetilde{u}-(\overline{u}+u^{0})\rangle_{\partial}+
  \langle L\overline{u}-(\mu_{j}+\Gamma)\overline{u},\widetilde{u}-(\overline{u}+u^{0})\rangle_{\partial}-\langle\Gamma u^{0},\widetilde{u}-(\overline{u}+u^{0})\rangle_{\partial}$$
$$=\langle L\widetilde{u}-(\mu_{j}+\Gamma)(\widetilde{u}),\widetilde{u}\rangle_{\partial}
-\langle L\widetilde{u},(\overline{u}+u^{0})\rangle_{\partial}+$$
$$\langle(\mu_{j}+\Gamma)\widetilde{u},\overline{u}+u^{0}\rangle_{\partial}+\langle L\overline{u}-(\mu_{j}+\Gamma)\overline{u},\widetilde{u}-(\overline{u}+u^{0})\rangle_{\partial}-\langle\Gamma u^{0},\widetilde{u}-(\overline{u}+u^{0})\rangle_{\partial}$$

$$=\langle L\widetilde{u}-(\mu_{j}+\Gamma)(\widetilde{u}),\widetilde{u}\rangle_{\partial}
-\langle L\widetilde{u},(\overline{u}+u^{0})\rangle_{\partial}+\langle(\mu_{j}+\Gamma)\widetilde{u},\overline{u}+u^{0}\rangle_{\partial}+ \langle L\overline{u}-(\mu_{j}+\Gamma)\overline{u},\widetilde{u}-(\overline{u}+u^{0})
$$$$\rangle_{\partial}-\langle\Gamma u^{0},\widetilde{u}\rangle_{\partial}+\langle\Gamma u^{0},\overline{u}\rangle_{\partial}
+\langle\Gamma u^{0},u^{0}\rangle_{\partial}$$

$$=\langle L\widetilde{u}-(\mu_{j}+\Gamma)(\widetilde{u}),\widetilde{u}\rangle_{\partial}
-\langle L\widetilde{u},\overline{u}\rangle_{\partial}-
\langle L\widetilde{u},u^{0}\rangle_{\partial}+
\langle(\mu_{j}+\Gamma)\widetilde{u},\overline{u}+u^{0}\rangle_{\partial}+ 
\langle L\overline{u}-(\mu_{j}+\Gamma)\overline{u},\widetilde{u}-(\overline{u}+u^{0})
\rangle_{\partial}-\langle\Gamma u^{0},\widetilde{u}\rangle_{\partial}+$$
$$\langle\Gamma u^{0},\overline{u}\rangle_{\partial}+\langle\Gamma u^{0},u^{0}\rangle_{\partial}$$

$$=\langle L\widetilde{u}-(\mu_{j}+\Gamma)(\widetilde{u}),\widetilde{u}\rangle_{\partial}
-\langle L\widetilde{u},\overline{u}\rangle_{\partial}-
\langle L\widetilde{u},u^{0}\rangle_{\partial}+
\mu_{j}\langle\widetilde{u},\overline{u}+u^{0}\rangle_{\partial}+\langle\Gamma\widetilde{u},\overline{u}+u^{0}\rangle_{\partial}+ 
\langle L\overline{u},\widetilde{u}\rangle_{\partial}-
\langle L\overline{u},\overline{u}+u^{0}\rangle_{\partial}
$$$$-\langle(\mu_{j}+\Gamma)\overline{u},\widetilde{u}-(\overline{u}+u^{0})\rangle_{\partial}
-\langle\Gamma u^{0},\widetilde{u}\rangle_{\partial}+$$
$$\langle\Gamma u^{0},\overline{u}\rangle_{\partial}+\langle\Gamma u^{0},u^{0}\rangle_{\partial}$$

$$=\langle L\widetilde{u}-(\mu_{j}+\Gamma)(\widetilde{u}),\widetilde{u}\rangle_{\partial}
-\cancel{\cancel{\langle L\widetilde{u},\overline{u}\rangle_{\partial}}}-
\langle L\widetilde{u},u^{0}\rangle_{\partial}+
\mu_{j}\langle\widetilde{u},\overline{u}+u^{0}\rangle_{\partial}+\langle\Gamma\widetilde{u},\overline{u}+u^{0}\rangle_{\partial}+ 
\cancel{\cancel{\langle L\overline{u},\widetilde{u}\rangle_{\partial}}}-
\langle L\overline{u},\overline{u}\rangle_{\partial}-\langle L\overline{u},u^{0}\rangle_{\partial}
$$$$-\langle(\mu_{j}+\Gamma)\overline{u},\widetilde{u}-(\overline{u}+u^{0})\rangle_{\partial}
-\langle\Gamma u^{0},\widetilde{u}\rangle_{\partial}+$$
$$\langle\Gamma u^{0},\overline{u}\rangle_{\partial}+\langle\Gamma u^{0},u^{0}\rangle_{\partial}$$
$$=\langle L\widetilde{u}-(\mu_{j}+\Gamma)(\widetilde{u}),\widetilde{u}\rangle_{\partial}
-
\langle L\widetilde{u},u^{0}\rangle_{\partial}+
\mu_{j}\langle\widetilde{u},\overline{u}+u^{0}\rangle_{\partial}+\langle\Gamma\widetilde{u},\overline{u}+u^{0}\rangle_{\partial} 
-\langle L\overline{u},\overline{u}\rangle_{\partial}-\langle L\overline{u},u^{0}\rangle_{\partial}
$$$$-\langle(\mu_{j}+\Gamma)\overline{u},\widetilde{u}-(\overline{u}+u^{0})\rangle_{\partial}
-\langle\Gamma u^{0},\widetilde{u}\rangle_{\partial}+$$
$$\langle\Gamma u^{0},\overline{u}\rangle_{\partial}+\langle\Gamma u^{0},u^{0}\rangle_{\partial}$$

$$=\langle L\widetilde{u}-(\mu_{j}+\Gamma)(\widetilde{u}),\widetilde{u}\rangle_{\partial}
-
\langle L\widetilde{u},u^{0}\rangle_{\partial}+
\mu_{j}\cancel{\langle\widetilde{u},\overline{u}+u^{0}\rangle_{\partial}}+
\langle\Gamma\widetilde{u},\overline{u}+u^{0}\rangle_{\partial} 
-
\langle L\overline{u},\overline{u}\rangle_{\partial}-\langle L\overline{u},u^{0}\rangle_{\partial}
$$$$-\langle(\mu_{j}+\Gamma)\overline{u},\widetilde{u}-(\overline{u}+u^{0})\rangle_{\partial}
-\langle\Gamma u^{0},\widetilde{u}\rangle_{\partial}+$$
$$\langle\Gamma u^{0},\overline{u}\rangle_{\partial}+\langle\Gamma u^{0},u^{0}\rangle_{\partial}$$

$$=\langle L\widetilde{u}-(\mu_{j}+\Gamma)(\widetilde{u}),\widetilde{u}\rangle_{\partial}
-\langle L\widetilde{u},u^{0}\rangle_{\partial}+
\langle\Gamma\widetilde{u},\overline{u}+u^{0}\rangle_{\partial} 
-\langle L\overline{u},\overline{u}\rangle_{\partial}-\langle L\overline{u},u^{0}\rangle_{\partial}
$$$$-\langle(\mu_{j}+\Gamma)\overline{u},\widetilde{u}-(\overline{u}+u^{0})\rangle_{\partial}
-\langle\Gamma u^{0},\widetilde{u}\rangle_{\partial}+$$
$$\langle\Gamma u^{0},\overline{u}\rangle_{\partial}+\langle\Gamma u^{0},u^{0}\rangle_{\partial}$$

$$=\langle L\widetilde{u}-(\mu_{j}+\Gamma)(\widetilde{u}),\widetilde{u}\rangle_{\partial}
-\langle L\widetilde{u},u^{0}\rangle_{\partial}+
\langle\Gamma\widetilde{u},\overline{u}+u^{0}\rangle_{\partial} 
-\langle L\overline{u},\overline{u}\rangle_{\partial}-\langle L\overline{u},u^{0}\rangle_{\partial}
$$$$-\mu_{j}\langle\overline{u},\widetilde{u}-(\overline{u}+u^{0})\rangle_{\partial}-
\langle\Gamma\overline{u},\widetilde{u}-(\overline{u}+u^{0})\rangle_{\partial}
-\langle\Gamma u^{0},\widetilde{u}\rangle_{\partial}+$$
$$\langle\Gamma u^{0},\overline{u}\rangle_{\partial}+\langle\Gamma u^{0},u^{0}\rangle_{\partial}$$

$$=\langle L\widetilde{u}-(\mu_{j}+\Gamma)(\widetilde{u}),\widetilde{u}\rangle_{\partial}
-\langle L\widetilde{u},u^{0}\rangle_{\partial}+
\langle\Gamma\widetilde{u},\overline{u}+u^{0}\rangle_{\partial} 
-\langle L\overline{u},\overline{u}\rangle_{\partial}-\langle L\overline{u},u^{0}\rangle_{\partial}
$$$$-\mu_{j}\langle\overline{u},\widetilde{u}-(\overline{u}+u^{0})\rangle_{\partial}-\langle\Gamma\overline{u},\widetilde{u}-(\overline{u}+u^{0})\rangle_{\partial}
-\langle\Gamma u^{0},\widetilde{u}\rangle_{\partial}+$$
$$\langle\Gamma u^{0},\overline{u}\rangle_{\partial}+\langle\Gamma u^{0},u^{0}\rangle_{\partial}$$

$$=\langle L\widetilde{u}-(\mu_{j}+\Gamma)(\widetilde{u}),\widetilde{u}\rangle_{\partial}
-\langle \widetilde{u},Lu^{0}\rangle_{\partial}+
\langle\Gamma\widetilde{u},\overline{u}+u^{0}\rangle_{\partial} 
-\langle L\overline{u},\overline{u}\rangle_{\partial}-\langle \overline{u},Lu^{0}\rangle_{\partial}
$$$$-\mu_{j}\langle\overline{u},\widetilde{u}\rangle_{\partial}+\mu_{j}\langle\overline{u},\overline{u}\rangle_{\partial}+
\mu_{j}\langle\overline{u},u^{0}\rangle_{\partial}-\langle\Gamma\overline{u},\widetilde{u}-(\overline{u}+u^{0})\rangle_{\partial}
-\langle\Gamma u^{0},\widetilde{u}\rangle_{\partial}+$$
$$\langle\Gamma u^{0},\overline{u}\rangle_{\partial}+\langle\Gamma u^{0},u^{0}\rangle_{\partial}$$

$$=\langle L\widetilde{u}-(\mu_{j}+\Gamma)(\widetilde{u}),\widetilde{u}\rangle_{\partial}
-\mu_{j}\cancel{\langle \widetilde{u},u^{0}\rangle_{\partial}}+
\langle\Gamma\widetilde{u},\overline{u}+u^{0}\rangle_{\partial} 
-\langle L\overline{u},\overline{u}\rangle_{\partial}-\mu_{j}\cancel{\langle \overline{u},u^{0}\rangle_{\partial}}
$$$$-\mu_{j}\cancel{\langle\overline{u},\widetilde{u}\rangle_{\partial}}+\mu_{j}\langle\overline{u},\overline{u}\rangle_{\partial}+
\mu_{j}\cancel{\langle\overline{u},u^{0}\rangle_{\partial}}-\langle\Gamma\overline{u},\widetilde{u}-(\overline{u}+u^{0})\rangle_{\partial}
-\langle\Gamma u^{0},\widetilde{u}\rangle_{\partial}+$$
$$\langle\Gamma u^{0},\overline{u}\rangle_{\partial}+\langle\Gamma u^{0},u^{0}\rangle_{\partial}$$

$$=\langle L\widetilde{u}-(\mu_{j}+\Gamma)(\widetilde{u}),\widetilde{u}\rangle_{\partial}
+
\langle\Gamma\widetilde{u},\overline{u}\rangle_{\partial}+\langle\Gamma\widetilde{u},u^{0}\rangle_{\partial} 
-\langle L\overline{u},\overline{u}\rangle_{\partial}
$$$$+\mu_{j}\langle\overline{u},\overline{u}\rangle_{\partial}
-\langle\Gamma\overline{u},\widetilde{u}\rangle_{\partial}+\langle\Gamma\overline{u},(\overline{u}+u^{0})\rangle_{\partial}
-\langle\Gamma u^{0},\widetilde{u}\rangle_{\partial}+$$
$$\langle\Gamma u^{0},\overline{u}\rangle_{\partial}+\langle\Gamma u^{0},u^{0}\rangle_{\partial}$$

$$=\langle L\widetilde{u}-(\mu_{j}+\Gamma)(\widetilde{u}),\widetilde{u}\rangle_{\partial}
+\cancel{\left(\langle\Gamma\widetilde{u},\overline{u}\rangle_{\partial}-\langle\Gamma\overline{u},\widetilde{u}\rangle_{\partial}\right)}
+\cancel{\left(\langle\Gamma\widetilde{u},u^{0}\rangle_{\partial}-\langle\Gamma u^{0},\widetilde{u}\rangle_{\partial}\right)} 
-\langle L\overline{u},\overline{u}\rangle_{\partial}
$$$$+\mu_{j}\langle\overline{u},\overline{u}\rangle_{\partial}
+\langle\Gamma\overline{u},(\overline{u}+u^{0})\rangle_{\partial}
+\langle\Gamma u^{0},\overline{u}\rangle_{\partial}+\langle\Gamma u^{0},u^{0}\rangle_{\partial}$$

$$=\langle L\widetilde{u}-(\mu_{j}+\Gamma)(\widetilde{u}),\widetilde{u}\rangle_{\partial}
-\langle L\overline{u},\overline{u}\rangle_{\partial}
+\mu_{j}\langle\overline{u},\overline{u}\rangle_{\partial}
$$ $$+\langle\Gamma\overline{u},(\overline{u}+u^{0})\rangle_{\partial}
+\langle\Gamma u^{0},\overline{u}\rangle_{\partial}+\langle\Gamma u^{0},u^{0}\rangle_{\partial}$$

$$D_{\Gamma}(u)=\langle L\widetilde{u}-(\mu_{j}+\Gamma)(\widetilde{u}),\widetilde{u}\rangle_{\partial}
-\langle L\overline{u},\overline{u}\rangle_{\partial}+\mu_{j}\langle\overline{u},\overline{u}\rangle_{\partial}
+\langle\Gamma(\overline{u}+u^{0}),\overline{u}+u^{0}\rangle_{\partial}$$

Since $\Gamma(x)$ is nonegative for $a.e.;~x\in\partial\Omega$ the last term is nonegative so we have 
$$D_{\Gamma}(u)\geq\langle L\widetilde{u}-(\mu_{j}+\Gamma)(\widetilde{u}),\widetilde{u}\rangle_{\partial}
-\langle L\overline{u},\overline{u}\rangle_{\partial}+\mu_{j}\langle\overline{u},\overline{u}\rangle_{\partial}
$$
By Parseval-Steklov identity (\cite{T}), we have that 
$$\mu_{j}\langle\overline{u},\overline{u}\rangle_{\partial}-\langle L\overline{u},\overline{u}\rangle_{\partial}$$
Since $L=\displaystyle\sum_{i=1}^{\infty}\mu_{i}P_{i}u$ and $\overline{u}=\displaystyle\sum_{1\leq i<N}P_{i}u$ so that 
$$\mu_{j}\langle\overline{u},\overline{u}\rangle_{\partial}-\langle L\overline{u},\overline{u}\rangle_{\partial}
=\mu_{j}\displaystyle\sum_{1\leq i<j}|P_{i}u|^{2}-\displaystyle\sum_{1\leq i<j}\mu_{i}|P_{i}u|^{2}=
\displaystyle\sum_{1\leq i<j}(\mu_{j}-\mu_{i})|P_{i}u|^{2}$$
By theorem \ref{th} we know that $(\mu_{j}-\mu_{i}) > 0$ whenever $i<j$ it clearly in case when $i=1$ then $(\mu{j}-\mu_{1})>0$
this implies that 
$$\displaystyle\sum_{1\leq i<j}(\mu_{j}-\mu_{i})|P_{i}u|^{2}\geq\displaystyle\sum_{1\leq i<j}[\min_{i}(\mu_{j}-\mu_{i})]|P_{i}u|^{2}$$
so that
\begin{equation}\label{13.5}
\begin{gathered}
\mu_{j}\langle\overline{u},\overline{u}\rangle_{\partial}-\langle L\overline{u},\overline{u}\rangle_{\partial}\geq \delta_{1}||\overline{u}||^{2}_{\partial}
\end{gathered}
\end{equation}

Where $$\delta_{1}=\mu_{j}-\mu_{j-1}> 0$$
Now, we show that there exists $\delta_{2}=\delta_{2}(\Gamma)$, such that
\begin{equation}\label{13.6}
\begin{gathered}
\langle L\widetilde{u}-(\mu_{j}+\Gamma)(\widetilde{u}),\widetilde{u}\rangle_{\partial}\geq\delta_{2}||\widetilde{u}||^{2}_{\partial}
\end{gathered}
\end{equation}

since we have that $\Gamma(x)\leq\mu_{j+1}-\mu_{j}$ for $a.e.;~x\in\partial\Omega$ one has 
$$\langle L\widetilde{u}-(\mu_{j}+\Gamma)(\widetilde{u}),\widetilde{u}\rangle_{\partial}\geq
\langle L\widetilde{u},\widetilde{u}\rangle_{\partial}-\mu_{j+1}\langle \widetilde{u},\widetilde{u}\rangle_{\partial}$$
sine we have that 
Since $L=\displaystyle\sum_{i=1}^{\infty}\mu_{i}P_{i}u$ and
$\widetilde{u}=\displaystyle\sum_{N<i<\infty}P_{i}u$

so we get that 
\begin{equation}\label{11}
\begin{gathered}
\langle L\widetilde{u}-(\mu_{j}+\Gamma)(\widetilde{u}),\widetilde{u}\rangle_{\partial}\geq
\displaystyle\sum_{j+1<i<\infty}(\mu_{i}-\mu_{j+1})|P_{i}u|^{2}
\end{gathered}
\end{equation}

since $\mu_{i}-\mu_{j+1}\geq0 ~~\forall~~j+1<i<\infty$
Therefore, $\langle L\widetilde{u}-(\mu_{j}+\Gamma)(\widetilde{u}),\widetilde{u}\rangle_{\partial}\geq0$
with equality if and only if $\widetilde{u}=\varphi_{j+1}$ with $\varphi_{j+1}\in N(L-\mu_{j+1}I).$
Hence, by  the assumption 
 $$\int_{\partial\Omega}(\mu_{j+1}-\mu_{j})\varphi_{j+1}^{2}(x)dx>0$$, we have claim
 \begin{claim}
 $\langle L\widetilde{u}-(\mu_{j}+\Gamma)(\widetilde{u}),\widetilde{u}\rangle_{\partial}=0$ if and only if $\widetilde{u}=0$
  \end{claim}
 \begin{proof}
 If $\widetilde{u}=0$, clearly that $\langle L\widetilde{u}-(\mu_{j}+\Gamma)(\widetilde{u}),\widetilde{u}\rangle_{\partial}=0$
 Now if $\langle L\widetilde{u}-(\mu_{j}+\Gamma)(\widetilde{u}),\widetilde{u}\rangle_{\partial}=0$
 $$\langle \mu_{j+1}\widetilde{u}-(\mu_{j}+\Gamma)(\widetilde{u}),\widetilde{u}\rangle_{\partial}=
 \int_{\partial\Omega}(\mu_{j+1}-\mu_{j}-\Gamma)\widetilde{u}^{2}=0$$ 
 Since we have that $\Gamma(x)<(\mu_{j+1}-\mu_{j}),$ 
 on a subset of $\partial\Omega$ of positive measure,  so that $\mu_{j+1}-\mu_{j}-\Gamma>0$
 so that $\widetilde{u}^{2}=0$, Therefore  $\widetilde{u}=0$
  \end{proof}
  $$\langle L\widetilde{u}-(\mu_{j}+\Gamma)(\widetilde{u}),\widetilde{u}\rangle_{\partial}\geq\delta_{2}||\widetilde{u}||^{2}_{\partial}$$
 Now assume the above relation is not true, then there is a sequence $\{\widetilde{u}_{n}\}_{n=1}^{\infty}\subset  \widetilde{H}(\partial\Omega)\cup Dom(L) $
 where $\widetilde{H}(\partial\Omega):=\{y\in W^{2}_{p}(\Omega):y=\displaystyle\sum_{N<i<\infty}P_{i}y\}$ such that 
 $||\widetilde{u}_{n}||_{\partial}=1~~\forall n\in\mathbb{N}$ and 
 \begin{equation}\label{12}
\begin{gathered}
\langle L\widetilde{u}_{n}-(\mu_{j}+\Gamma)(\widetilde{u}_{n}),\widetilde{u}_{n}\rangle_{\partial}\leq\frac{1}{n}
\end{gathered}
\end{equation}
 
 Now we write 
$ \widetilde{H}=N(L-\mu_{j+1}I)\oplus_{\partial}\widetilde{H}^{1},$ where $N(L-\mu_{j+1}I)$ is the finite-dimensional eigenspace and 
$\widetilde{H}^{1}$ is orthogonal (in $\widetilde{H}$) to $N(L-\mu_{j+1}I).$ It is clear that 
$$\widetilde{u}_{n}=w_{n}+v_{n}$$ with $w_{n}\in N(L-\mu_{j+1}I)$ and $v_{n}\in\widetilde{H}^{1}$.
Using inequalities (\ref{11}) and (\ref{12}),it follows that $v_{n}\to 0$ in $H^{1}(\partial\Omega)$ as $n\to\infty$.
Now $N(L-\mu_{j+1}I)$ is finite-dimensional (see theorem \ref{th}) and since 
$1=||\widetilde{u}_{n}||^{2}_{\partial}=||w_{n}||^{2}_{\partial}+||v_{n}||_{\partial}^{2},$  we have a subsequence of $\{w_{n}\},$
which we many relabel as  $\{w_{n}\},$ converges strongly to same $w\in N(L-\mu_{j+1}I)$ with $||w||_{\partial}=1$, 
consequently, 
$$\frac{1}{n}\geq\langle L\widetilde{u}_{n}-(\mu_{j}+\Gamma)(\widetilde{u}_{n}),\widetilde{u}_{n}\rangle_{\partial}=
\langle Lw_{n}-(\mu_{j}+\Gamma){w}_{n},w_{n}\rangle_{\partial}
-2\langle (\mu_{j}+\Gamma){w}_{n},v_{n}\rangle_{\partial}
+\langle Lv_{n}-(\mu_{j}+\Gamma){v}_{n},v_{n}\rangle_{\partial}$$
we know that $-\Gamma\geq(-\mu_{j+1}+\mu_{j})$ and $v_{n}\in\widetilde{H}^{1}$
$$\langle Lv_{n}-(\mu_{j}+\Gamma){v}_{n},v_{n}\rangle_{\partial}=(\mu_{j+2}-\mu_{j+1})||v_{n}||^{2}_{\partial}$$

so we have that 
$$\frac{1}{n}\geq\langle L\widetilde{u}_{n}-(\mu_{j}+\Gamma)(\widetilde{u}_{n}),\widetilde{u}_{n}\rangle_{\partial}=
\langle Lw_{n}-(\mu_{j}+\Gamma){w}_{n},w_{n}\rangle_{\partial}
-2\langle (\mu_{j}+\Gamma){w}_{n},v_{n}\rangle_{\partial}
+\langle Lv_{n}-(\mu_{j}+\Gamma){v}_{n},v_{n}\rangle_{\partial}$$
$$\geq\langle (\mu_{j+1}-(\mu_{j}+\Gamma)){w}_{n},w_{n}\rangle_{\partial}
-2\langle (\mu_{j}+\Gamma){w}_{n},v_{n}\rangle_{\partial}
+(\mu_{j+2}-\mu_{j+1})||v_{n}||^{2}_{\partial}$$
Using $v_{n}\to 0$ and $w_{n}\to w$ as $n\to\infty$, one obtains
$$0\geq\langle Lw_{n}-(\mu_{j}+\Gamma){w}_{n},w_{n}\rangle_{\partial}\to
\int_{\partial\Omega} (\mu_{j+1}-(\mu_{j}+\Gamma){w}_{n}^{2}\,dx$$
and since $\Gamma(x)\leq\mu_{j+1}-\mu_{j}$ for $a.e.;~x\in\partial\Omega$  $\Gamma(x)<(\mu_{j+1}-\mu_{j}),$ 
 on a subset of $\partial\Omega$ of positive measure,  so that $\mu_{j+1}-\mu_{j}-\Gamma>0$ one has 
$$0=\int_{\partial\Omega} (\mu_{j+1}-(\mu_{j}+\Gamma){w}_{n}^{2}\,dx~~{\rm with}~~w\in N(L-\mu_{j+1})$$
So that, by the assumption (\ref{13}), one has $w=0$. A contradiction with $||w||_{\partial}=1.$ Therefore, inequality
(\ref{14}) is proven.\\
Choosing $\delta=min\{\delta_{1},\delta_{2}\}$ and observing that 
$$||{u^{\bot}}||^{2}_{\partial}=||\overline{u}||^{2}_{\partial}+||\widetilde{u}||^{2}_{\partial}$$
Therefore,
$$D_{\Gamma}(u):=\langle Lu-(\mu_{j}+\Gamma)u,\widetilde{u}-(\overline{u}+u^{0})\rangle_{\partial}\geq\delta||{u^{\bot}}||^{2}_{H^{1}(\partial\Omega)}$$
the proof is complete.
\end{proof}

\begin{lemma}\label{le2}
Let $\Gamma\in L^{\infty}(\partial\Omega)$ be as in lemma\ref{le1} and $\delta>0$ be associated to $\Gamma$ by that lemma.
Let $\epsilon>0$. Then, for all $p\in L^{\infty}(\partial\Omega)$ satisfying

 \begin{equation}\label{14}
\begin{gathered}
0\leq p(x)\leq\Gamma(x)+\epsilon
\end{gathered}
\end{equation}
 
 $a.e.;$ on $\partial\Omega$ and all $u\in Dom(L)$, one has
 
 \begin{equation}\label{15}
\begin{gathered}
D_{p}(u):=\langle Lu-(\mu_{j}+\Gamma)u,\widetilde{u}-(\overline{u}+u^{0})\rangle_{\partial}
\geq (\delta-\epsilon)||{u^{\bot}}||^{2}_{H^{1}(\partial\Omega)}
\end{gathered}
\end{equation}
  \end{lemma}
  \begin{proof}
  If $u\in Dom(L)$, then using computations of lemma\ref{le1}, we obtain 
  
  \begin{equation}\label{16}
\begin{gathered}
D_{p}(u)=\langle L\widetilde{u}-(\mu_{j}+p)\widetilde{u},\widetilde{u}\rangle_{\partial} +
\mu_{j}\langle\overline{u},\overline{u}\rangle_{\partial}
-\langle L\overline{u},\overline{u}\rangle_{\partial}
+\langle p(\overline{u}+u^{0}),\overline{u}+u^{0}\rangle_{\partial}
\\ \geq \langle L\widetilde{u}-(\mu_{j}+\Gamma)\widetilde{u},\widetilde{u}\rangle_{\partial}
+\mu_{j}\langle\overline{u},\overline{u}\rangle_{\partial}
-\langle L\overline{u},\overline{u}\rangle_{\partial}-\epsilon||\widetilde{u}||^{2}_{\partial}
\end{gathered}
\end{equation}
  
 Therefore, by the inequalities  (\ref{13.6}) and (\ref{13.5}) one has 
 
  \begin{equation}\label{17}
\begin{gathered}
D_{p}(u)\geq (\delta-\epsilon)||{u^{\bot}}||^{2}_{H^{1}(\partial\Omega)},
\end{gathered}
\end{equation}
and the proof is complete.
   \end{proof}
   \begin{lemma}\label{le3}
   Let $q\in(0,\mu_{j+1}-\mu_{j})$ be fixed, then, exists a constant $\eta>0,$ such that for all $u\in Dom(L),$ one has 
   $$||\frac{\partial u}{\partial\nu}-\mu_{j}u-qu||_{L^{2}(\partial\Omega)}\geq\eta||u||_{H^{2}}$$
   \end{lemma}
   \begin{proof}
   By the theory of the linear first order differential equations \cite{A, Bz,MWW}, 
   the operator $$E:Dom(L)\to L^{2}(\partial\Omega)$$ defined by 
   $$Eu:=\frac{\partial u}{\partial\nu}-\mu_{j}u+qu$$
   Clearly 
   $KerE=\{0\}$ so, $E$ is one-to-one, onto and obivously continuous. It follows that $E^{-1}:L^{2}(\partial\Omega)\to Dom(L)$
   is linear and continuous\cite{Bz}. Taking $\eta\leq\frac{1}{||E^{-1}||}$ 
   
   \begin{remark}\color{red}{we know that 
   $$||E^{-1}||=\displaystyle\sup_{u\neq 0}\frac{||E^{-1}(u)||_{H^2}}{||u||_{L^{2}(\partial\Omega)}}$$
   so 
   $$\frac{||E^{-1}(u)||_{H^2}}{||u||_{L^{2}(\partial\Omega)}}\leq||E^{-1}||$$
   since you have taking 
   $$\eta\leq\frac{1}{||E^{-1}||}\leq\frac{||u||_{L^{2}(\partial\Omega)}}{||E^{-1}(u)||_{H^2}}$$
   so we have that 
    $$\eta||E^{-1}(u)||_{H^2}\leq ||u||_{L^{2}(\partial\Omega)} $$ 
    Since $u\in L^{2}(\partial\Omega)\exists y: ~u=Ey=\frac{\partial y}{\partial\nu}+\mu_{j}y+qy$ and $E^{-1}u=y$
    Therefore, 
    $$\eta||y)||_{H^2}\leq ||\frac{\partial y}{\partial\nu}+\mu_{j}y+qy||_{L^{2}(\partial\Omega)}$$}
   \end{remark} 
 The proof is complete
 \end{proof}  
 The following lemma is essentially due to De Figueredo\cite{FS} in the entire space we will make new version for the boundary, the proof is similar to the proof in \cite{IK2} replace $[0,2\pi]$
 by $\partial\Omega$
 \begin{lemma}\label{le4}
Let  $g\colon\partial\Omega\times\mathbb{R}\rightarrow\mathbb{R}$
be a function verifying  Carath\'{e}odory conditions and satisfying the following conditions
\begin{description}
\item[i] There exist functions $a,A\in L^{2}(\partial\Omega)$ and constants $R_{1},R_{2}$ with 
$R_{1}<0<R_{2},$ such that 
$$g(x,u)\geq A(x)$$
for $a.e.;~x\in \partial\Omega$ and all $u\geq R_{2},$ 
$$g(x,u)\leq a(x)$$
for $a.e.;~x\in \partial\Omega$ and all $u\leq R_{1}.$
\item[ii] There exist functions $b,c\in L^{2}(\partial\Omega)$ and a constant $B\geq 0$ such that 
$$g(x,u)\leq c(x)|u|+b(x)$$
for $a.e.;~x\in \partial\Omega$ and all $u\geq B,$ \\
\textbf{Then}, \\
for each real number $k>0,$ there is decomposition 
 \begin{equation}\label{18.5}
\begin{gathered}
g(x,u)=q_{k}(x,u)+g_{k}(x,u)
\end{gathered}
\end{equation}
of g by functions $q{k},$ and $g_{k}$ verifying Carath\'{e}odory conditions and satisfying the following conditions
\begin{equation}\label{18.6}
\begin{gathered}
0\leq uq_{k}(x,u),~~0\leq ug_{k}(x,u)
\end{gathered}
\end{equation}
for $a.e.;~x\in \partial\Omega$ and all $u\in \mathbb{R},$
\begin{equation}\label{18.7}
\begin{gathered}
|q_{k}(x,u)|\leq c(x)|u|+b(x)+k
\end{gathered}
\end{equation}
for $a.e.;~x\in \partial\Omega$ and all $u$ with $|u|\geq\max(1,B),$ there is a function 
$\sigma_{k}\in L^{2}(\partial\Omega)$ depending on $a,A,$ and $g$ such that 
\begin{equation}\label{18.8}
\begin{gathered}
|g_{k}(x,u)|\leq\sigma_{k}
\end{gathered}
\end{equation}
for $a.e.;~x\in \partial\Omega$ and all $u\in \mathbb{R},$
 \end{description}
 
 \end{lemma}
 Assume that the function $g:\partial\Omega\times\mathbb{R}\to\mathbb{R}$ satisfies 
 Carath\'{e}odory conditions and grows at most linearly $i.e.;$
 \begin{equation}\label{000}
\begin{gathered}
|g(x,u)|\leq d|u|+e(x)
\end{gathered}
\end{equation}
 
 for some constant $d\geq 0$, some $e\in L^{2}(\partial\Omega)~a.e.x\in \partial\Omega$ and all $u\in\mathbb{R}.$ by those assumptions
  now we can defined the nonlinear (Nemystk\v{i}i) operator 
 $$\v{N}:W_{p}^{1-\frac{1}{p}}(\partial\Omega)\subset C(\partial\Omega)\to W_{p}^{1-\frac{1}{p}}(\partial\Omega) $$
 by 
 $$\v{N}u:=g(.,u(.))$$
 We shall consider solvability of the equation  ( we will add and subtrac ($\mu_{j}u)$)
 \begin{equation}\label{18.9}
\begin{gathered}
Lu-\mu_{j}u-\v{N}u+\mu_{j}u=h~\quad\text{ }\forall u\in Dom(L)
\end{gathered}
\end{equation}
Eq(\ref{E1}) is then equivalent to (\ref{18.9})
 \begin{proof}
 Proof of Theorem\ref{thm1}. Let $\delta>0$ be associated to the function $\Gamma$ by Lemma\ref{le1}. Then, by the assumption
 \ref{02}, there exist $B(\delta)=B>0$ and $b=b(\delta)\in L^{\infty}(\partial\Omega),$ such that 
 \begin{equation}\label{18}
\begin{gathered}
|g(x,u)|\leq(\Gamma(x)+\frac{\delta}{4})|u|+b(x),
\end{gathered}
\end{equation}
for $a.e.;~x\in \partial\Omega$, and all $u\in\mathbb{R},$ with $|u|\geq B.$ 
Useing Lemma\ref{le4} with $k=1$, equation (\ref{18.9} is then equivalent to  
\begin{equation}\label{18.10}
\begin{gathered}
Lu-\mu_{j}u-q_{1}(.,u(.))-g_{1}(.,u(.))+\mu_{j}u=h~\quad\text{ }\forall u\in Dom(L)
\end{gathered}
\end{equation}
Where $q_{1},g_{1}$ are Carath\'{e}odory functions satisfying conditions (\ref{18.6}) and (\ref{18.9}). 
Moreover by (\ref{18.7})
\begin{equation}\label{18.00}
\begin{gathered}
|q_{1}(x,u)|\leq (\Gamma(x)+\frac{\delta}{4})|u|+b(x)+1
\end{gathered}
\end{equation}
for $a.e.;~x\in \partial\Omega$ and all $u$ with $|u|\geq\max(1,B),$
Let us choose $\bar{B}>\max(1,B)$ such that 

\begin{equation}\label{18.17.1}
\begin{gathered}
\frac{(b(x)+1)}{|u|}<\frac{\delta}{4}
\end{gathered}
\end{equation}
for $a.e.;~x\in \partial\Omega$ and all $u$ with $|u|\geq\bar B,$. It follows (\ref{18.00}) and (\ref{18.17.1}), one has 

\begin{equation}\label{18.17.2}
\begin{gathered}
0\leq u^{-1}q_{1}(x,u)\leq\Gamma(x)+\frac{\delta}{2}
\end{gathered}
\end{equation}

for $a.e.;~x\in \partial\Omega$ and all $u$ with $|u|\geq\bar B,$.

 Let us define 
$$\widetilde{\gamma}:\partial\Omega\times\mathbb{R}\to\mathbb{R}$$ by 
\begin{displaymath}
   \widetilde{\gamma}(x,u) = \left\{
     \begin{array}{ccc}
       u^{-1}q_{1}(x,u)                                        & \quad\text{for}&|u|\geq \bar B \\
       \bar B^{-1}q_{1}(x,\bar B)(\frac{u}{\bar B})+(1-\frac{u}{\bar B})\Gamma(x)  &\quad\text{for}&0\leq u\leq \bar B\\
       \bar B^{-1}q_{1}(x,-K)(\frac{u}{\bar B})+(1+\frac{u}{\bar B})\Gamma(x) & \quad\text{for}& -\bar B\leq u\leq 0
     \end{array}
   \right.
\end{displaymath}
Then, by assumption (\ref{18.6}) and the relation (\ref{18.17.2}), we have 
\begin{equation}\label{19}
\begin{gathered}
0\leq\widetilde{\gamma}(x,u)\leq \Gamma(x)+\frac{\delta}{2}
\end{gathered}
\end{equation}
for $a.e.;~x\in \partial\Omega$, and all $u\in\mathbb{R}.$  Moreover the function $\widetilde{\gamma}(x,u)u$ satisfies
Carath\'{e}odory condition and $$f:\partial\Omega\times\mathbb{R}\to\mathbb{R}$$ defined  by 

\begin{equation}\label{20}
\begin{gathered}
f(x,u):=g_{1}(x,u)+q_{1}(x,u)-\widetilde{\gamma}(x,u)u,
\end{gathered}
\end{equation}

is such that for $a.e.;~x\in \partial\Omega$, and all $u\in\mathbb{R}.$
\begin{equation}\label{21.010}
\begin{gathered}
|f(x,u)|\leq v(x)
\end{gathered}
\end{equation}
For some $v\in L^{2}(\partial\Omega)$ dependent only on $\Gamma$ and $\gamma_{K}$ given by (\ref{000})

Now Let $$h=-H$$

Therefore, equation (\ref{E1},\ref{18.9},\ref{18.10}) is equivalent $\forall u\in Dom(L)$
\begin{equation}\label{21}
\begin{gathered}
Lu-\mu_{j}u-\widetilde{\gamma}(.,u(.))u-f(.,u(.))+\mu_{j}u=-H(.)~~\forall u\in Dom(L) ~~a.e.; x\in\partial\Omega
\end{gathered}
\end{equation}
to which we shall apply Mawhin's continuation theorem \cite{M1}
Let us define
$$G:W_{p}^{1-\frac{1}{p}}(\partial\Omega)\subset C(\partial\Omega)\to W_{p}^{1-\frac{1}{p}}(\partial\Omega)$$ by
$$Gu=\widetilde\gamma(.,u(.)u(.)+f(.,u(.))-H(x)$$
$$A:W_{p}^{1-\frac{1}{p}}(\partial\Omega)\subset C(\partial\Omega)\to W_{p}^{1-\frac{1}{p}}(\partial\Omega)$$ by 
$$Au=\frac{\delta}{2}u(.)$$
Equation (\ref{21}) is equivalent to solving 
\begin{equation}\label{21.1}
\begin{gathered}
Lu-\mu_{j}u-Gu+\mu_{j}u=0
\end{gathered}
\end{equation}
in $Dom(L)$\\ 
If $N(L)$ is finite dimensional, it is clear that $L$ is a linear Fredholm of index zero see \cite{MN} and $G$ and $A$ are well defined
and $L-compact$ on bounded of $W_{p}^{1-\frac{1}{p}}(\partial\Omega)$. By theorem $IV.12$ in \cite{M1}, equation (\ref{21.1}) will have a solution
if we can show that for any $\lambda\in[0,1)$ and any $u\in Dom(L)$ such that 
\begin{equation}\label{21.2}
\begin{gathered}
Lu-\mu_{j}u-(1-\lambda)Au-\lambda Gu+\mu_{j}u=0
\end{gathered}
\end{equation}
one has $||u||_{C^{1}(\partial\Omega)}<K_{0}$ (for some constant $K_{0}>0$ independent of $\lambda$ and $u$)
If $u\in Dom(L)$ satisfies (\ref{21.2}) for some $\lambda\in[0,1)$, then one has 
\begin{equation}\label{21.3}
\begin{gathered}
Lu(x)-\mu_{j}u(x)-[(1-\lambda)\frac{\delta}{2}+\lambda\widetilde\gamma(x,u(x))]u(x)-\lambda Gu+\lambda H(x)+\mu_{j}u(x)=0
\end{gathered}
\end{equation}
with, by (\ref{19})
$$0\leq(1-\lambda)\frac{\delta}{2}+\lambda\widetilde\gamma(x,u(x))\leq\Gamma(x)+\frac{\delta}{2}$$
for $a.e.;~x\in \partial\Omega$
\begin{remark}\label{rem1}

$$\mu_{j}\langle\widetilde{u}-(\bar{u}+u^{0}),u(.)\rangle_{\partial}=
\mu_{j}(\langle\widetilde{u},\widetilde{u}\rangle_{\partial}-\langle\bar{u},\bar{u}\rangle_{\partial}-\langle{u^{0}},{u^{0}}\rangle_{\partial}$$
$$\geq \mu_{j}(-\langle\bar{u},\bar{u}\rangle_{\partial}-\langle{u^{0}},{u^{0}}\rangle_{\partial})$$
By Cauchy Schwarz inequality
$$\mu_{j}\langle\widetilde{u}-(\bar{u}+u^{0}),u(.)\rangle_{\partial}\geq-(||\widetilde{u}||_{H^{1}(\partial\Omega)}+||\bar{u}||_{H^{1}(\partial\Omega)}
+||{u^{0}}||_{H^{1}(\partial\Omega)})$$
\end{remark}

It is clear for $\lambda=0,$ equation  (\ref{21.2}) has only the trivial $i.e.;~u=0$ solution in $Dom(L)$. Now if $u\in Dom(L)$ is solution of 
(\ref{21.2}) foe some $\lambda\in[0.1)$, then using lemma \ref{le2}, Cauchy Schwarz inequality,theorem \ref{th} and Remark \ref{rem1}  we get 

$$0=\langle \widetilde{u}-(\bar{u}+u^{0}),Lu(x)-[\mu_{j}+(1-\lambda)\frac{\delta}{2}+\lambda\widetilde\gamma(.,u(.))]u(.) \rangle_{\partial}
+\langle\widetilde{u}-(\bar{u}+u^{0}),\lambda H(.)-f(.,u(.)\rangle_{\partial}+\mu_{j}\langle\widetilde{u}-(\bar{u}+u^{0}),u(.)\rangle_{\partial}\geq
$$
$$ \frac{\delta}{2}||{u^{\bot}}||^{2}_{H^{1}(\partial\Omega)}-(||\widetilde{u}||_{H^{1}(\partial\Omega)}+||\bar{u}||_{H^{1}(\partial\Omega)}
+||{u^{0}}||_{H^{1}(\partial\Omega}))\big(\mu_{j}+||h||_{L^{2}(\partial\Omega)}+||f(.,u(.)||_{L^{2}(\partial\Omega)}\big)$$
$$
\geq \frac{\delta}{2}||{u^{\bot}}||^{2}_{H^{1}(\partial\Omega)}-(||\widetilde{u}||_{H^{1}(\partial\Omega)}+||\bar{u}||_{H^{1}(\partial\Omega)}
+||{u^{0}}||_{L^{2}(\partial\Omega}))\big(\mu_{1}+||h||_{L^{2}(\partial\Omega)}+||f(.,u(.)||_{L^{2}(\partial\Omega)}\big)
$$
and by inequality(\ref{21.010}) we have 
\begin{equation}\label{21.4}
\begin{gathered}
0 \geq \frac{\delta}{2}||{u^{\bot}}||^{2}_{H^{1}(\partial\Omega)}-\beta\big(||{u^{\bot}}||^{2}_{H^{1}(\partial\Omega)}+||{u^{0}}||_{H^{1}(\partial\Omega)}\big)
\end{gathered}
\end{equation}
 For some constant $\beta>0$ dependent only on $\mu_{1},~v$, and $H$ (but not on $u$ or $\lambda$. So that, 
 taking $\alpha=\beta(\delta)^{-1}$ we have 
 \begin{equation}\label{21.5}
\begin{gathered}
||{u^{\bot}}||^{2}_{H^{1}(\partial\Omega)}\leq \alpha+\sqrt{\alpha^{2}+2\alpha||{u^{0}}||_{H^{1}(\partial\Omega)}}
\end{gathered}
\end{equation}
\begin{claim}
There exist a constant $K_{0}>0$ such that 
$||u||_{C^{1}(\partial\Omega)}<K_{0}$
for all $u\in Dom(L)$ satisfies (\ref{21.2}) $K_{0}>0$ independent of $\lambda$ and $u$)
\end{claim}
Proof of the claim. Assume that the claim does not hold. Then, there will be sequence $\{\lambda_{n}\}_{n=1}^{\infty}$
in the open interval $(0,1)$, and $\{u_{n}\}_{n=1}^{\infty}$ in ${C^{1}(\partial\Omega)}$ with
$||u_{n}||_{C^{1}(\partial\Omega)}\geq n$  (in which  $q\in (0,\mu_{j+1}-\mu_{j})$ with $q=\frac{\delta}{2}$ fixed  such that 
\begin{equation}\label{25}
\begin{gathered}
Lu_{n}+(1-\lambda_{n})qu_{n}-\lambda_{n}g(x,u_{n})=\lambda_{n}h.
\end{gathered}
\end{equation}
Let $v_{n}=\frac{u_{n}}{||u_{n}||_{C^{1}(\partial\Omega)}},$ we have
$$||u_{n}||_{C^{1}(\partial\Omega)}(Lv_{n}+qv_{n})=
\lambda_{n}h+\lambda_{n}qu_{n}+\lambda_{n}g(x,u_{n})$$
\begin{equation}\label{25.1}
\begin{gathered}
Lv_{n}+qv_{n}=
\frac{\lambda_{n}h}{||u_{n}||_{C^{1}(\partial\Omega)}}+\lambda_{n}qv_{n}+\frac{\lambda_{n}g(x,u_{n})}{||u_{n}||_{C^{1}(\partial\Omega)}}
\end{gathered}
\end{equation}
or 
\begin{equation}\label{25.01}
\begin{gathered}
\frac{\partial v_{n}}{\partial \nu}-\mu_{j}+qv_{n}=
\frac{\lambda_{n}h}{||u_{n}||_{C^{1}(\partial\Omega)}}+\lambda_{n}qv_{n}+\frac{\lambda_{n}g(x,u_{n})}{||u_{n}||_{C^{1}(\partial\Omega)}}
\end{gathered}
\end{equation}
 or equivalent 
 \begin{equation}\label{25.2}
\begin{gathered}
Ev_{n}=
\frac{\lambda_{n}h}{||u_{n}||_{C^{1}(\partial\Omega)}}+\lambda_{n}qv_{n}+\frac{\lambda_{n}g(x,u_{n})}{||u_{n}||_{C^{1}(\partial\Omega)}}
\end{gathered}
\end{equation}
Where $E:Dom(L)\subset C^{1}(\partial\Omega)\to L^{2}(\partial\Omega)$ is defined by 
$Ev=Lv+qv.$ According to Lemma \ref{le3} and compact embedding of $Dom(L)$ into $C^{1}(\partial)$ \cite{Bz}, $E$ is invertible and $E^{-1}$ 
is compact (completely continuous) as an operator from $L^{2}(\partial\Omega)$ into $C^{1}(\partial\Omega)$. On the other hand, by inequality  
(\ref{00}) and the growth condition \ref{02}, it follows that there exists a function $c\in L^{2}(\partial\Omega)$ depending only in 
$R=R(\delta)>0$ such that 
$$|g(x,u)|\leq(\Gamma(x)+q)|u|+b(x)+c(x)$$
for $a.e.;~x\in\partial\Omega$ and all $u\in\mathbb{R}$. so that, the sequence 
$\frac{g(x,u_{n})}{||u_{n}||_{C^{1}(\partial\Omega)}}$ is bounded in $L^{2}(\partial\Omega)$. 
Hence the right-hand member of equality (\ref{25.2})
is bounded in $L^{2}(\partial\Omega)$ independent of $n$. Therefore, writing equation (\ref{25.2}) in the equivalent form 
\begin{equation}\label{25.3}
\begin{gathered}
v_{n}=
E^{-1}\big[\frac{\lambda_{n}h}{||u_{n}||_{C^{1}(\partial\Omega)}}+\lambda_{n}qv_{n}+\frac{\lambda_{n}g(x,u_{n})}{||u_{n}||_{C^{1}(\partial\Omega)}}\big]
\end{gathered}
\end{equation}
and using the compactness of $E^{-1}:L^{2}(\partial\Omega)to C^{1}(\partial\Omega)$ we can assume (going if necsessary to a subsequence relabeled
$v_{n}$), that there exists $v\in C^{1}(\partial\Omega)$ such that $v_{n}\to v$ in $C^{1}(\partial\Omega)$ as $n\to\infty$, 
$||v||_{C^{1}(\partial\Omega)}=1$ and $v\in Dom(L)$ 
on the other hand using inequality \ref{21.4} or \ref{21.5} one deduces that $v_{n}^{\bot}\to o$ in $H^{1}(\partial\Omega)$. Therefore 
$v\in\widetilde{H}^{1}(\partial\Omega)$ i.e 
$$v(x)=A\widetilde{u}=A\widetilde{u}=\displaystyle\sum_{N<j<\infty}P_{j}u=A\displaystyle\sum_{N<j<\infty}\varphi_{j}$$
choose that $$||\widetilde{u}||^{C^{1}(\partial\Omega)}=1$$ %$$||\displaystyle\sum_{N<j<\infty}\varphi_{j}||=1$$
for some $A\in\mathbb{R}$ Since $||v||_{C^{1}(\partial\Omega)}=1$ so that $A=\pm 1$ In what follows, we shall suppose that 
$v(x)=\widetilde{u}$ (the case $v(x)=-\widetilde{u}$ is treated in a similar way).
Now, using the fact that $v_{n}\to v$ in ${C^{1}(\partial\Omega)}$ and with since  $\widetilde u_{n}\in\ker{L}$   so 
$$\frac{\partial\widetilde{u}_n}{\partial\nu}-\mu_{j}\widetilde u_{n}\to 0.$$ So that for $n\geq n_{0}$
$v_{n}(x)>0$ for $a.e.;~x\in\partial\Omega$ 
so, 
\begin{equation}\label{25.4}
\begin{gathered}
u_{n}>0~~~u_{n}\in Dom(L)
\end{gathered}
\end{equation}
Now writing 
$$v_{n}=\bar{v}_{n}+v_{n}^{0}+\widetilde{v}_n$$
we have that 
$\bar{v}_{n}=A_{n} bar{v}=A_{n}\displaystyle\sum_{1<j<N}\varphi_{j},$
${v}_{n}^{0}=B_{n}v^{0}=B_{n}\varphi_{N},$ 
$\widetilde{v}_n=C_{n}\widetilde{v}=C_{n}\displaystyle\sum_{N<j<\infty}\varphi_{j}$
Let us look back to equation (\ref{25.01}). Taking the inner product in ($L^{2}(\partial\Omega))$ of (\ref{25.01})
with $\widetilde{v}_{n}$, remarking that $\lambda_{n}\in(0,1)$ and considering assumption (\ref{04}), we deduce that 
$$\frac{\lambda_{n}}{||u_{n}||_{C^{1}(\partial\Omega)}}\int_{\partial\Omega}g(x,u_{n}(x))(\widetilde{v}_{n})dx<0$$ for all $n$ 
sufficiently large so $\int_{\partial\Omega}g(x,u_{n}(x))(\widetilde{v})dx<0$ this is a contradiction, since by (\ref{25.4}) and assumption 
\ref{01} one has the $g(x,u_{n}(x))(\widetilde{v})\geq 0$ on $x\in\partial\Omega$ for $n\geq n_{0}$, 
and the proof is complete.

 \end{proof}
 
\begin{example}
Let $\Omega\subset\mathbb{R}^N$ , $N\geq 2$ is a bounded domain  with boundary
$\partial \Omega$ of class $C^2$ and  $\partial\Omega=A\cup B,$ consider equation 
\begin{equation}\label{Ex1}
\begin{gathered}
 -\Delta u + c(x)u = 0 \quad\text{in } \Omega,\\
\frac{\partial u}{\partial \nu}=\mu_{1}u+g(x,u)+h(x)  \quad\text{on }
\partial\Omega,
\end{gathered}
\end{equation}
Where $\mu_{1}$ the first eigenvalue (\ref{E2}), and $g:\partial\Omega\times\mathbb{R}\to\mathbb{R}$ is defined  by 
\begin{displaymath}
   g(x,u) = \left\{
     \begin{array}{ccc}
       \mu_{1}u(x)\sin^{2}(u(x))            & \forall x\in A\\
       0                            & \forall x\in A\cap B \\
       0                          & \forall x\in B      
     \end{array}
   \right.
\end{displaymath}
It is seen that that all the assumptions of Theorem \ref{thm1} are fulfilled So that equation (\ref{Ex1})
has at least one solution for any $h\in L^{2}(\partial\Omega)$ with 
$$\int_{\partial\Omega}h(x)\varphi_{1}(x)dx=0$$ where $\varphi_{1}$ the first eigenfunction of (\ref{E2})
Obvously $g(x,.)$ is dos not satisfy the Landesman-Lazer conditions since $$\limsup_{u\to-\infty}g(x,u)=\liminf_{u\to\infty}g(x,u)=0$$
Notice that $g$ is unbounded
\end{example}
 \begin{remark} If you consider the problem 
 \begin{equation}\label{FE}
\begin{gathered}
 -\Delta u + c(x)u = f(x,u) \quad\text{in } \Omega,\\
\frac{\partial u}{\partial \nu}=\mu_{j}u+g(x,u)+h(x)  \quad\text{on }
\partial\Omega,
\end{gathered}
\end{equation}
Step by step the approach in \ref{thm1} with obivous modifications in the $$Dom(L):=\{u\in W^{2}_{p}(\Omega):-\Delta u+c(x)u-f(x,u)=0\}$$
and the notation, Then Eq(\ref{FE}) has at least one solution. \cite{A}  
 \end{remark}

\end{document}